\documentclass{amsart}
\usepackage{amsmath,amsfonts,amssymb,amsthm, enumerate,latexsym,mathrsfs, color}
\usepackage[alphabetic]{amsrefs}
\usepackage{graphicx}
\usepackage[all]{xy}
\newcounter{nthm}

\newcounter{nthmm}

\newtheorem{thm}{Theorem}

\newtheorem{thm*}[nthm]{Theorem}
\newtheorem{thm**}[nthmm]{Th\'eor\`eme}
\newtheorem{defn}{Definition}
\newtheorem*{defn*}{Definition}
\newtheorem{prop}{Proposition}
\newtheorem{prop*}[nthm]{Proposition}

\newtheorem{cor}{Corollary}

\newcommand{\A}{\mathcal A}
\newcommand{\N}{\mathbb{N}}
\newcommand{\Z}{\mathbb{Z}}

\newcommand{\Q}{\mathbb{Q}}

\newcommand{\K}{\mathcal{K}}

\newcommand{\C}{\mathbb{C}}

\newcommand{\restrict}[2]{#1 \upharpoonright #2}

\newcommand{\m}[1]{\textbf{#1}}

\newcommand{\e}{\varepsilon}

\newcommand{\Aut}{\mathrm{Aut}}
\newcommand{\Stab}{\mathrm{Stab}}
\newcommand{\Age}{\mathrm{Age}}

\newcommand{\actson}{\curvearrowright}

\newcommand{\RUCB}{\mathrm{RUC}_{b}}
\newcommand{\ROB}{\mathrm{Ro}_{b}}
\newcommand{\WAP}{\mathrm{WAP}}
\newcommand{\AP}{\mathrm{AP}}

\author{Lionel Nguyen Van Th\'e}

\address{Aix Marseille Univ, CNRS, Centrale Marseille, I2M UMR 7373, 13453 Marseille, France}

\email{lionel.nguyen-van-the@univ-amu.fr}

\thanks{This work has been partially supported by the GrupoLoco project (ANR-11-JS01-0008) funded by the French Government, managed by the French National Research Agency (ANR)}



\date{December 2017}

\title[Fixed points in compactifications and combinatorial counterparts]{Fixed points in compactifications and combinatorial counterparts}

\begin{document}
\maketitle

\begin{abstract}
The Kechris-Pestov-Todorcevic correspondence connects extreme amenability of topological groups with Ramsey properties of classes of finite structures. The purpose of the present paper is to recast it as one of the instances of a more general construction, allowing to show that Ramsey-type statements actually appear as natural combinatorial expressions of the existence of fixed points in certain compactifications of groups, and that similar correspondences in fact exist in various dynamical contexts. 
\end{abstract}

\section{Introduction and results} 

\label{section:Intro}

\subsection{Introduction}

In \cite{KPT}, Kechris, Pestov and Todorcevic established a striking correspondence between topological dynamics and structural Ramsey theory (for a precise statement, see Theorem \ref{thm:KPT} below). Building on the seminal works of Graham-Rothschild \cite{GR}, Graham-Leeb-Rothschild \cites{GLR1, GLR2}, Abramson-Harrington \cite{AH} and Ne\v set\v ril-R\"odl \cites{NR1,NR2}, this turned out to be an invaluable tool to produce extremely amenable groups when concentration of measure is not available (as in \cites{GM,Gl1, GP}), and to reach a better understanding of the dynamics of infinite-dimensional topological groups (see for example \cites{AKL, Z1} in the non-Archimedean Polish case, or \cites{MT, MNT, BYMT} in the general Polish case). It also considerably impacted the recent activity around 
around Fra\"iss\'e theory and structural Ramsey theory, providing new incentives to construct and/or identify highly homogeneous structures (see \cites{KuS, Ku, EFHKKL}), and to prove and/or use new partition results (see for example the paper \cite{Sol6} and references therein, the surveys \cite{Bod2} and \cite{NVT6}, as well as \cites{BaK2, BaK3, BLLM, HN, PS} for more recent results).  

The purpose of this paper is to recast the Kechris-Pestov-Todorcevic correspondence as an instance of a more general construction, allowing to show that Ramsey-type statements actually appear naturally when expressing combinatorially the existence of fixed points in certain compactifications of groups. As a consequence, it is proved in a unified way that similar correspondences in fact exist in various dynamical contexts. Some of them are presented here as illustrations, and exhibit combinatorial properties that are equivalent or implied by fixed-point properties like minimal almost periodicity, strong amenability and amenability. Among those, some isolate new phenomena, while some others allow to recover some previously known results that were originally obtained in different contexts.

The original motivation to undertake such a project was also to gain a better understanding of those non-Archimedean Polish groups that contain a coprecompact extremely amenable subgroup. According to \cite{MNT} and \cite{Z2}, this class coincides with those groups for which all minimal flows are metrizable (and have a generic orbit). It also captures the so-called \emph{Ramsey property}, which expresses a particularly good behaviour from the point of view of partition calculus, and whose distribution remains mysterious among classes of finite structures. The new connections that are established in the present work do not solve that problem, but somehow make more precise the contours of the ``dark'' side that remains to be understood when attacking it ``from below''. 

\subsection{Results}

Throughout this paper, a \emph{$G$-flow} of a topological group $G$ will be a continuous action of $G$ on some compact Hausdorff space, while a \emph{$G$-ambit} will be a $G$-flow together with a distinguished point whose orbit is dense. These objects will be referred to via the following notation: $G\actson X$ for $G$-flows, $G\actson(X,x)$ for $G$-ambits. Unless explicitly specified, all actions will be \emph{left actions}. 

The main line of attack, whose initial part shares some features with \cite[Section 11]{GlM2} by Glasner-Megrelishvili (though this was realized only \emph{a posteriori}), builds on the works of Pestov \cites{Pe1, Pe0, Pe} and of Kechris-Pestov-Todorcevic \cite{KPT}, and can be condensed as follows: Assume that some class $\mathcal X$ of $G$-flows admits a \emph{universal $G$-ambit}
in the sense that every $G$-ambit with underlying flow in $\mathcal X$ is a factor of $G\actson (X,x)$. Such an object always appears 
as the Gelfand compactification $G^{\A}$ of $G$ with respect to a particular C$^{*}$-algebra $\A$, which can be described explicitly. Modulo certain technical requirements, there is a fixed point in every $G$-flow in $\mathcal X$ iff $G\actson G^{\A}$ has a fixed point. This last fact can be expressed in terms of a property on the elements of $\A$, first isolated by Pestov, and called \emph{finite oscillation stability}. Under appropriate assumptions, this property discretizes as a Ramsey-type statement, which can sometimes be completely finitized. 

This strategy leads to the master result of this paper, Theorem \ref{thm:fpRP}, and will be particularized to the following classes of flows (where arrows symbolize inclusions): 

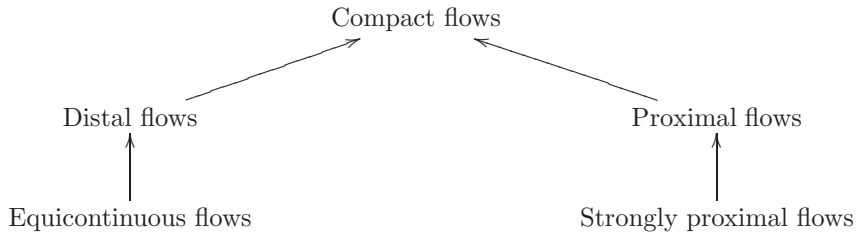
\[ \xymatrix{
& \textrm{Compact flows} & \\
\textrm{Distal flows} \ar[ru] & & \textrm{Proximal flows} \ar[lu]\\
\textrm{Equicontinuous flows} \ar[u] & & \textrm{Strongly proximal flows} \ar[u]
} \]

Recall that given a flow $G\actson X$ and $x,y\in X$, the ordered pair $(x,y)$ is \emph{proximal} when there exists a net $(g_{\alpha})_\alpha$ of elements of $G$ such that $\lim_\alpha g_\alpha \cdot x = \lim_\alpha g_\alpha \cdot y$. Otherwise, $(x,y)$ is \emph{distal}. Equivalently, these notions can be expressed in terms of the uniformity $\mathcal U_{X}$ of $X$:  $(x,y)$ is proximal when for every $U$ in $\mathcal U_{X}$, there exists $g\in G$ so that $(g\cdot x,g\cdot y)\in U$; $(x,y)$ is distal when there is $U$ in $\mathcal U_{X}$ so that no $g\in G$ satisfies $(g\cdot x,g\cdot y)\in U$. Then, the flow $G\actson X$ is \emph{proximal} when every $(x,y)\in X^{2}$ is proximal, \emph{strongly proximal} when the induced flow on the Borel probability measures of $X$ is proximal, and \emph{distal} when every $(x,y)\in X^{2}$ with $x\neq y$ is distal. A strict subclass of the distal flows is provided by the \emph{equicontinuous flows} which satisfy: $$\forall U_{\varepsilon}\in \mathcal U_{X} \ \exists U_{\eta}\in \mathcal U_{X} \ \forall x,y\in X \quad (x,y)\in U_{\eta} \Rightarrow \forall g\in G \ (g\cdot x, g\cdot y)\in U_{\varepsilon}$$

To each of the aforementioned classes of flows, one can associate a natural fixed-point property: a topological group $G$ is \emph{extremely amenable} when every $G$-flow has a fixed point, \emph{strongly amenable} when every proximal $G$-flow has a fixed point, \emph{amenable} when every strongly proximal $G$-flow has a fixed point (equivalently, every $G$-flow admits a $G$-invariant Borel probability measure), and \emph{minimally almost periodic} when every equicontinuous $G$-flow has a fixed point (which is known to be equivalent to having a fixed point on any \emph{distal} $G$-flow, having no non-trivial finite-dimensional unitary representation, and/or admitting no non-trivial continuous morphism to a compact group). This leads to the following ``dual'' form of the previous diagram: 

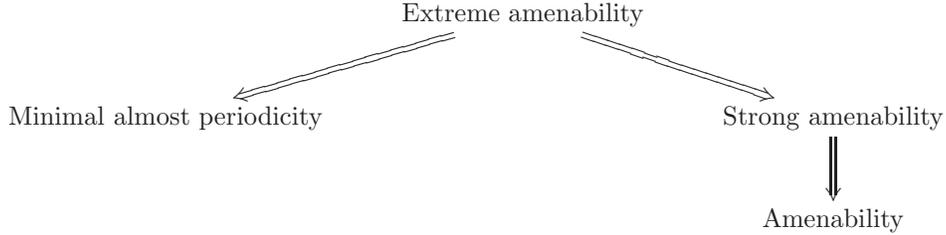
\[ \xymatrix{
& \textrm{Extreme amenability} \ar@{=>}[ld] \ar@{=>}[rd]& \\
\textrm{Minimal almost periodicity}& & \textrm{Strong amenability} \ar@{=>}[d]\\
& & \textrm{Amenability}
} \] 

In practice, the aforementioned strategy suggests in fact two slightly different kinds of applications. Starting from a natural class $\mathcal X$ of flows, one may express combinatorially the fixed point property relative to those flows; this requires some particular conditions on $\mathcal X$, which are satisfied for equicontinuous/distal flows and for proximal flows. Conversely, starting from natural algebras, one may isolate a class of flows on which the fixed point property is combinatorially meaningful. This will be done for the Roelcke algebra, and to some extent for the weakly almost periodic algebra. The relationship between all the corresponding ambits can be represented as follows, where $S(G)$ stands for the Samuel compactification of $G$, $R(G)$ for the Roelcke compactification, $W(G)$ for the weakly almost periodic compactification, $B(G)$ for the Bohr compactification, $P(G)$ for the proximal compactification, and $P_{S}(G)$ for the strongly proximal compactification: 

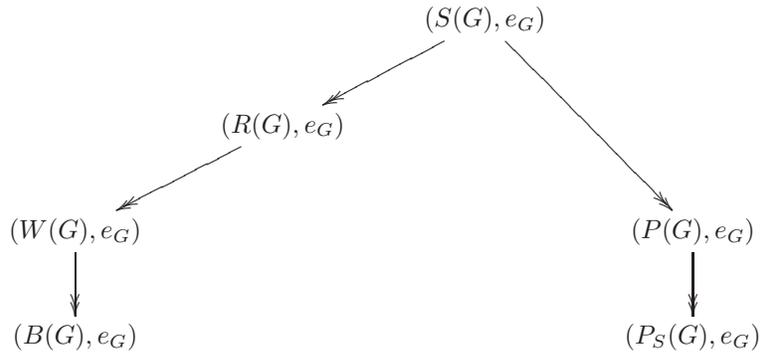
\[ \xymatrix{
& & (S(G),e_{G}) \ar@{->>}[ld] \ar@{->>}[rdd]& \\
& (R(G),e_{G}) \ar@{->>}[ld] & & \\
(W(G),e_{G}) \ar@{->>}[d] & &  & (P(G),e_{G}) \ar@{->>}[d]\\
(B(G),e_{G})& & & (P_{S}(G),e_{G})
} \]

\

On the combinatorial side, the general setting is that of first-order structures in the usual sense of model-theory (see for example \cite{H} for a standard reference) but for simplicity, we will restrict our attention to the relational setting. Given a first-order relational language (i.e. a family $(R_{i})_{i\in I}$ of symbols together with associated arities $m_{i}\geq 1$), a \emph{structure} $\m A$ is a non-empty set $A$, together with a family of subsets $R^{\m A}_{i}\subset A^{m_{i}}$ for every $i\in I$. To such objects is naturally attached a notion of isomorphism and of embedding, where an embedding is just an isomorphism onto its image; given two structures $\m A$ and $\m B$, the set of all embeddings of $\m A$ in $\m B$ will be denoted by $\binom{\m B}{\m A}$ (note that this differs from the common notation according to which $\binom{\m B}{\m A}$ refers to the set of all substructures of $\m B$ isomorphic to $\m A$). A structure is \emph{ultrahomogeneous} when any isomorphism between any two of its finite substructures extends to an automorphism. There is now a rich theory around those objects, starting with the seminal work of Fra\"iss\'e himself \cite{Fr0}. For that reason, countable ultrahomogeneous structures are now called \emph{Fra\"iss\'e} structures (denoted by $\m F$ is the sequel). In the recent developments of Fra\"iss\'e theory, a main concern is the study of the interaction between the combinatorics of the set $\Age(\m F)$ of all finite substructures of $\m F$ and the dynamics of the automorphism group $\Aut(\m F)$. The main theorem of \cite{KPT} is a striking illustration of this: 

\begin{thm}[Kechris-Pestov-Todorcevic \cite{KPT}]
\label{thm:KPT}
Let $\m F$ be a Fra\"iss\'e structure. TFAE: 
\begin{enumerate}
\item[i)] $\Aut(\m F)$ is extremely amenable. 
\item[ii)] $\Age(\m F)$ has the Ramsey property. 
\end{enumerate}
\end{thm}

The Ramsey property (for embeddings) referred to in the previous results means that for every $\m A\in \Age(\m F)$, every function $\chi$ taking finitely many values on $\binom{\m F}{\m A}$ (such a $\chi$ is usually referred to as a finite coloring) is necessarily constant on arbitrarily large finite set. Precisely: given any $\m B \in \Age(\m F)$, in which $\m A$ typically embeds in many ways, $\chi$ is constant of some set of the form $\binom{b(\m B)}{\m A}$, for some $b\in \binom{\m F}{\m B}$. Under that form, the Ramsey property is a property of $\m F$ rather than $\Age(\m F)$, but it finitizes under the following form: for every $\m A, \m B\in \Age(\m F)$, every $k\in \N$, there exists $\m C\in \Age(\m F)$ such that every coloring of $\binom{\m C}{\m A}$ taking at most $k$ many values is constant on $\binom{b(\m B)}{\m A}$, for some $b\in \binom{\m C}{\m B}$. The typical result of the paper will be of similar flavor. Its general form, condensed in Theorem \ref{thm:fpRP}, states that $\Aut(\m F)$ has a fixed point property of a particular kind iff $\m F$ has some Ramsey-type property, restricted to some particular kind of colorings (see Section \ref{section:colorings} and \ref{section:ambits} for definitions): 


\begin{thm}
\label{thm:fpRP}
Let $\m F$ be a Fra\"iss\'e structure, $\mathcal X$ be a class of $\Aut(\m F)$-flows such that the class of $\mathcal X$-$\Aut(\m F)$-ambits is closed under suprema and factors, and that every $\Aut(\m F)\actson X\in \mathcal X$ admits some $x\in X$ such that $\Aut(\m F)\actson \overline{\Aut(\m F)\cdot x}\in \mathcal X$. Then $\A:=\{ f\in \RUCB(G) : G\actson \overline{G\bullet f} \in \mathcal X\}$ is a unital, left-invariant, closed C$^{*}$-subalgebra of $\RUCB(\Aut(\m F))$, and TFAE: 
\begin{enumerate}
\item[i)] Every $\Aut(\m F)$-flow in $\mathcal X$ has a fixed point. 
\item[ii)] For every $\e>0$, $\m F$ has the Ramsey property up to $2\e$ for the finite colorings in $(\A)_{\e}$. 
\end{enumerate}

Those imply the following equivalent statements:  
\begin{enumerate}
\item[iii)] Every zero-dimensional $\Aut(\m F)$-flow in $\mathcal X$ has a fixed point. 
\item[iv)] $\m F$ has the Ramsey property for the finite colorings in $\A$.
\end{enumerate} 

When the finite colorings are dense in $\A$, all those statements are equivalent.

\end{thm}


Notice that by considering the class $\mathcal X$ of \emph{all} $G$-flows, which obviously satisfies the hypotheses of Theorem \ref{thm:fpRP}, we directly obtain Theorem \ref{thm:KPT}. By varying the class of flows under consideration, this will lead to several other concrete incarnations. 

The left side of the diagrams of the previous page appears to be particularly well adapted for such an analysis. A \emph{joint embedding} $\langle a,z\rangle$ of two structures $\m A$ and $\m Z$ is a pair $(a,z)$ of embeddings of $\m A$ and $\m Z$ into some common structure $\m C$. Two such pairs $\langle a,z\rangle$ with common range $\m C$ and $\langle a',z'\rangle$ with common range $\m C'$ are \emph{isomorphic} (written $\langle a,z\rangle \cong \langle a',z'\rangle$) when there is an isomorphism $c:\m C\rightarrow \m C'$ so that $a'=c\circ a$ and $z'=c\circ z$. Occasionally, the isomorphim type of a joint embedding $\langle a,z\rangle$ will be referred to as its \emph{joint embedding pattern} and will be written $[a,z]$. In what follows, because the language is assumed to be relational, the joint embeddings which satisfy $C=a(A)\cup z(Z)$ will be the only ones to be considered, without any explicit mention of $\m C$. Note also that the notion of joint embedding $\langle \m A, \m Z^{1},...,\m Z^{k}\rangle$ and joint embedding pattern $[\m A, \m Z^{1},...,\m Z^{k}]$ can be defined in the same way in the case of finitely many structures $\m A$, $\m Z^{1}$,...,$\m Z^{k}$. 

\begin{defn}
Let $\K$ be a class of finite structures in some first order language. It has the \emph{definable Ramsey property} when for every $\m A, \m B\in \K$, every $\m Z\in \K$, there exists $\m C \in \K$ such that for every joint embedding  $\langle c,z\rangle$ of $\m C$ and $\m Z$, there is $b\in \binom{\m C}{\m B}$ so that the coloring $a\mapsto [a,z]$ is constant on $\binom{b(\m B)}{\m A}$. 
\end{defn}

Note the similarity with the usual Ramsey property. For the combinatorialist, what has just been defined should be thought as $\m C \rightarrow (\m B)^{\m A}_{\m Z}$. The dynamical meaning of the definable Ramsey property will be made explicit soon, but in view of the fixed point properties described previously, it makes more sense to consider first the following weakening, which will look familiar to the model theorist:

\begin{defn}
Let $\K$ be a class of finite structures in some first order language, and $\m A, \m Z\in \K$. An \emph{unstable $(\m A,\m Z)$-sequence} is a family of joint embeddings $(\langle a_{m}, z_{n}\rangle)_{m, n\in\N}$ of $\m A$ and $\m Z$ such that there exist two different joint embedding patterns $\tau_{<}$ and $\tau_{>}$ satisfying: $$ \forall m,n\in \N \quad (m<n \Rightarrow [a_{m},z_{n}] = \tau_{<})\wedge (m>n \Rightarrow [a_{m},z_{n}] = \tau_{>})$$


When there is no unstable $(\m A, \m Z)$-sequence, the pair $(\m A, \m Z)$ is \emph{stable}. 

\end{defn}

\begin{defn}
Let $\K$ be a class of finite structures in some first order language. It has the \emph{stable Ramsey property} when it has the definable Ramsey property in restriction to stable pairs. More formally: for every $\m A, \m B\in \K$, every $\m Z^{1},...,\m Z^{k}\in \K$ so that every pair $(\m A, \m Z^{i})$ is stable, there exists $\m C \in \K$ such that for every joint embedding  $\langle c,z^{1},...,z^{k}\rangle$, there is $b\in \binom{\m C}{\m B}$ so that for every $i\leq k$, the joint embedding pattern $[a,z^{i}]$ does not depend on $a\in \binom{b(\m B)}{\m A}$. 
\end{defn}

With these notions in mind, here is the characterization of minimal almost periodicity in the spirit of the Kechris-Pestov-Todorcevic correspondence: 

\begin{thm}
\label{thm:MAPStableRP}
Let $\m F$ be a Fra\"iss\'e structure with Roelcke-precompact automorphism group. TFAE: 
\begin{enumerate}
\item[i)] $\Aut(\m F)$ is minimally almost periodic.  
\item[ii)] For every $\m A\in \Age(\m F)$, every $\Aut(\m F)$-invariant equivalence relation on $\binom{\m F}{\m A}$ with finitely many classes is trivial. 
\item[iii)] $\Age(\m F)$ has the stable Ramsey property.
\end{enumerate} 
\end{thm}

This approach provides a new proof of the equivalence between the first two items, which already appears in the work of Tsankov \cite{Ts} where unitary representations of oligomorphic groups were classified, or of Ben Yaacov \cite{BY} where the relationship between the Bohr compactification and the algebraic closure of the empty set was already identified. Note also that since minimal almost periodicity is implied by the existence of a fixed point in the Roelcke compactification, it can also be proved thanks to the following: 

\begin{thm}
\label{thm:fpRoecke}
Let $\m F$ be a Fra\"iss\'e structure. TFAE:
\begin{enumerate}
\item[i)] The flow $\Aut(\m F)\actson R(\Aut(\m F))$ has a fixed point. 
\item[ii)] For every $\m A, \m B, \m Z\in \Age(\m F)$, every finite coloring $\gamma$ of the joint embedding patterns of $\m A$ and $\m Z$, there exists a joint embedding  $\langle b,z\rangle$ such that the coloring $a\mapsto \gamma([a,z])$ is constant on $\binom{b(\m B)}{\m A}$. 
\end{enumerate} 

When $\Aut(\m F)$ is Roelcke-precompact, these conditions are equivalent to: 
\begin{enumerate}
\item[iii)] For every $\m A, \m B, \m Z\in \Age(\m F)$, there exists a joint embedding  $\langle b,z\rangle$ such that the coloring $a\mapsto [a,z]$ is constant on $\binom{b(\m B)}{\m A}$.  
\end{enumerate}
\end{thm}

This result is very useful in practice; for example, it automatically holds when $\Age(\m F)$ has the free amalgamation property. Therefore, the automorphism group of the random graph, of the random hypergraph of any fixed finite type, or of any Henson graph (= countable ultrahomogeneous $K_{n}$-free graph for some $n\in \N$) is minimally almost periodic (this can also be proved using a different method, see \cite{NVT8}). As a slightly more involved application, Theorem \ref{thm:fpRoecke} can also be used to prove that the orthogonal group of $\ell^{2}$ is minimally almost periodic when equipped with its strong operator topology (see Section \ref{subsubsection:mapell2}). Much more is known about that object but the present proof is, in comparison, rather simple.  

Here is now the dynamical content of the definable Ramsey property: 

\begin{thm}
\label{thm:minRoeckeDefRP}
Let $\m F$ be a Fra\"iss\'e structure with Roelcke-precompact automorphism group. TFAE: 
\begin{enumerate}
\item[i)] Every minimal subflow of $\Aut(\m F)\actson R(\Aut(\m F))$ is trivial. 
\item[ii)] $\Age(\m F)$ has the definable Ramsey property.
\end{enumerate} 
\end{thm}

Besides those discussed above, Theorem \ref{thm:fpRoecke} and Theorem \ref{thm:minRoeckeDefRP} exhibit several interesting features, among which (see Section \ref{subsection:Remarks} for more details): the interaction between amalgamation properties and Ramsey properties (which was first isolated in the pioneering work of Ne\v set\v ril and R\"odl in \cite{NR2}); a distinction between the finite language case and the $\omega$-categorical case (this is connected to the problems mentioned in \cite{BPT}*{Section 7}); the possibility of a proof of it by induction, which sometimes reduces the task to a proof of elementary pigeonhole principles, in the spirit of \cite{T} and \cite{Sol4}; the model-theoretic flavor, which certainly calls for a deeper study in that direction. 

For the right side of the diagram from p.3, the general strategy applies as well, but the corresponding results turn out to be of a rather different flavor. 

\begin{defn}
\label{def:proxcoloring}
Let $\m F$ be a Fra\"iss\'e structure and $\chi$ be a coloring of $\binom{\m F}{\m A}$. Say that $\chi$ is \emph{proximal} when for every $\m D \in \Age(\m F)$, there exists $\m E \in \Age(\m F)$ such that for every $e_{1}, e_{2}\in \binom{\m F}{\m E}$, there exists $d\in \binom{\m E}{\m D}$ such that the colorings $a\mapsto \chi(e_{1}\circ a)$ and $a\mapsto \chi(e_{2}\circ a)$ agree on $\binom{d(\m D)}{\m A}$. 
\end{defn}

\begin{defn}
Let $\m F$ be a Fra\"iss\'e structure. Say that $\m F$ has the \emph{proximal Ramsey property} when for every $\m A, \m B \in \Age(\m F)$ and every finite proximal coloring $\chi$ of $\binom{\m F}{\m A}$, there is $b\in \binom{\m F}{\m B}$ such that $\chi$ is constant on $\binom{b(\m B)}{\m A}$. 
\end{defn}

\begin{thm}

\label{thm:SAProxRP}

Let $\m F$ be a Fra\"iss\'e structure. TFAE: 
\begin{enumerate}
\item[i)] Every zero-dimensional proximal $\Aut(\m F)$-flow has a fixed point. 
\item[ii)] $\m F$ has the proximal Ramsey property. 
\end{enumerate}
When the finite proximal colorings are uniformly dense in the set of all proximal functions, these statements are equivalent to $\Aut(\m F)$ being strongly amenable. (For the precise meaning of this last sentence, see Section \ref{section:colorings}.) 
\end{thm}

Theorem \ref{thm:SAProxRP} is, however, less satisfactory than the previous ones on the practical side, for at least two reasons. The first one is the intrusion of a non-trivial condition, of topological nature, which potentially truly limits the use of our combinatorial methods (see Section \ref{subsection:RksProx} for a more detailed discussion). The second one is that at the present stage, because of the difficulty to handle proximal colorings in concrete structures, there is no example where Theorem \ref{thm:SAProxRP} can be used to prove strong amenability by combinatorial means. It can, however, be used to deduce non-trivial combinatorial consequences from strong amenability.  

The same obstacles appear when considering amenability and strongly proximal flows. In fact, this case is, in some sense, even more resistant, as it remains unclear whether a combinatorial description of the relevant class of colorings in the spirit of Definition \ref{def:proxcoloring} exists at all. Nevertheless, a slight modification of the general strategy leads to the Ramsey-theoretic counterpart previously obtained by Moore and by Tsankov:  

\begin{defn}[Moore \cite{M}]

Let $\K$ be a class of finite structures in some first order language. It has the \emph{convex Ramsey property} when for every $\m A, \m B \in \K$, and every $\e>0$, there exists $\m C\in \K$ such that for every finite coloring $\chi$ of $\binom{\m C}{\m A}$, there is a finite convex linear combination $\lambda_{1},...,\lambda_{n}$, and $b_{1},...,b_{n}\in \binom{\m C}{\m B}$ such that the coloring $\displaystyle a\mapsto \sum_{i=1}^{n}\lambda_{i}\chi(b_{i}\circ a)$ is constant up to $\e$ on $\binom{\m B}{\m A}$. 
\end{defn}

\begin{thm}[Moore \cite{M}; Tsankov \cite{Ts1}]

\label{thm:AConvexRP}

Let $\m F$ be a Fra\"iss\'e structure. TFAE: \begin{enumerate}
\item[i)] $\Aut(\m F)$ is amenable. 
\item[ii)] $\Age(\m F)$ has the convex Ramsey property. 
\end{enumerate}
\end{thm}

The practical use of this result in order to study amenability is so far rather limited, but there are promising exceptions (see Section \ref{section:stronglyproximal} for a more detailed discussion). 

The paper is organized is follows: The first part is devoted to the proof of two master results, Theorem \ref{thm:fpRP} and Theorem \ref{thm:fpRPcpct}, of which all the previous results are specific incarnations. This proof is based on a general analysis of the existence of fixed points in compactifications of topological groups via the notion of finite oscillation stability (Section \ref{section:oscillationstability}) and on its discretization in Ramsey-theoretic terms (Section \ref{section:RPfp}). The second part of the paper focuses on applications. Section \ref{section:Roelcke} deals with the Roelcke algebra and the Roelcke compactification, leading to Theorems \ref{thm:fpRoecke} and \ref{thm:minRoeckeDefRP}. Equicontinuous and distal flows are treated in Section \ref{section:map}, leading to Theorem \ref{thm:MAPStableRP}. Proximal flows are discussed in Section \ref{section:proximal}, leading to Theorem \ref{thm:SAProxRP}. Strongly proximal flows and amenability are discussed in Section \ref{section:stronglyproximal}, leading to Theorem \ref{thm:AConvexRP}. 


As a final remark before starting: most of the present work can certainly be completed in the context of continuous Fra\"iss\'e theory, in the spirit of \cite{MT}. I leave it to the interested reader to make the appropriate translation. 

\section{Fixed points in compactifications and finite oscillation stability} 

\label{section:oscillationstability}

In this section, given a topological group $G$, the goal is to isolate conditions that characterize the existence of fixed points in certain compactifications of $G$. 

To do so, for the sake of completeness, certain general facts about uniformities on $G$ are first reminded (for a more detailed treatment, see for example \cite{Bou} or \cite{Eng}). Recall that such a structure is a family $\mathcal U$ of subsets of $G\times G$, often called \emph{entourages (of the diagonal)}, which satisfies the following properties:

\begin{enumerate}
\item Every $U\in \mathcal U$ contains the diagonal $\{ (g,g) \in G^{2}: g\in G \}$. 
\item The family $\mathcal U$ is closed under supersets and finite intersections. 
\item If $U$ is in $U$, so is $U^{-1}:=\{ (h,g) \in G^{2}: (g,h)\in U \}$. 
\item If $U\in \mathcal U$, there is $V\in \mathcal U$ so that $V\circ V \subset U$, where $V\circ V$ is the set $$V\circ V:=\{ (g,h) \in G^{2}: \exists k \in G \ \ (g,k)\in V \wedge (k,h)\in V\}$$
\end{enumerate}

Informally, when $(g,h)\in U$, $g$ and $h$ must be thought as $U$-close. Such a structure naturally appears when $G$ is equipped with a metric (in which case a typical entourage is of the form $\{ (g,h) \in G^{2}: d(g,h)<\e\}$ for some $\e>0$), but there is no need for a metric to have a uniformity. Uniform structures constitute the natural framework to express the concepts of uniform continuity and of completion. Here, uniformities will be useful because they will make it possible to manipulate various compactifications of $G$ while staying within $G$. More precisely, every compact topological space admits a unique compatible uniformity. Therefore, when $G$ is compactified (i.e. continuously mapped onto a dense subspace of a compact space), it inherits a natural uniformity, which retains all the information about the whole compactification. In particular, if $G$ acts on the compactification, detecting the existence of fixed points is possible when the interaction between the group operation on $G$ and the uniformity is well understood. 

In the present case, $G$ is not just a set but a topological group, and it carries several natural uniformities. A few of them are described below, starting with the left uniformity, the right uniformity, and the Roelcke uniformity. 

The left uniformity $\mathcal U_{L}$ is generated by those entourages of the form $$V_{U}=\{ (g,h)\in G^{2}: g^{-1}h\in U\}$$ where $U$ is an open neighborhood of the identity element $e_{G}$. It is induced by any left-invariant metric $d_{L}$ compatible with the topology of $G$ (of course, when $G$ is Polish, there is always such a metric). When $G$ is of the form $\Aut(\m F)$, where $\m F$ is a Fra\"iss\'e structure whose underlying set is $\N$, a basis of open neighborhoods of $e_{G}$ consists of clopen subgroups of the form $\Stab(\m A)$, where $\Stab(\m A)$ denotes the pointwise stabilizer of $\m A$, and $\m A\subset \m F$ denotes a finite substructure. In this uniformity, two elements $g, h\in G$ are $\m A$-close when $g$ and $h$ agree on $A$. Thus, the corresponding entourage can be seen as a partition of $G$ into sets of the form $g\Stab(\m A)$, and the corresponding quotient space coincides with the usual (algebraic) quotient $G/\Stab(\m A)$. 


The right uniformity $\mathcal U_{R}$ is defined in a similar way. It is generated by those sets of the form $$V_{U}=\{ (g,h)\in G^{2}: gh^{-1}\in U\}$$ where $U$ is an open neighborhood of the identity element $e_{G}$. It is induced by any right-invariant metric $d_{R}$ compatible with the topology of $G$. When $G$ is of the form $\Aut(\m F)$ and $\m A$ is a finite substructure of $\m F$, two elements $g, h\in G$ are $\m A$-close when $g^{-1}$ and $h^{-1}$ agree on $A$. The corresponding entourage can be seen as a partition of $G$ into sets of the form $\Stab(\m A)g$, and the corresponding quotient space coincides with the right quotient $\Stab(\m A)\backslash G$. For reasons that will become clear later on, we will see in detail in Section \ref{section:colorings} how to think about these objects. 

%

The Roelcke uniformity $\mathcal U_{L}\wedge\mathcal U_{R}$ is the finest uniformity that is coarser than the two previous uniformities. When $G$ is of the form $\Aut(\m F)$, a typical uniform neighborhood of this uniformity is indexed by two finite substructures $\m A, \m Z \subset \m F$, and two elements $g,h\in G$ are $(\m A,\m Z)$-close when they are equal in the double quotient $\Stab(\m A)\backslash G / \Stab(\m Z)$. We will see in Section \ref{section:Roelcke} how to translate this combinatorially. 


So far, uniformities were given via a description of their entourages. For those that are induced by compactifications of $G$, another convenient way to produce them is to use algebras of bounded functions. For example, consider the set $\RUCB(G)$ of all bounded uniformly continuous maps from $(G,d_{R})$ to $\C$ (these maps are called \emph{right-uniformly continuous}). This is a unital C$^{*}$-algebra when equipped with the supremum norm, on which the group $G$ acts continuously by left shift: $g\cdot f(h)=f(g^{-1}h)$. In what follows, we will follow the terminology from \cite[Chapter IV, Section 5]{dV} and will call \emph{left-invariant} the closed C$^{*}$-subalgebras of $\RUCB(G)$ that are invariant under this action. 

Given a $G$-flow $G\actson X$, and $x\in X$, there is a very simple way to produce such an object. Let $\mathrm C(X)$ denote the space of (bounded) continuous functions from $X$ to $\C$. This is a unital C$^{*}$-algebra when equipped with the supremum norm. For $f\in \mathrm C(X)$, define the map $f_{x}:G\rightarrow \C$ by $f_{x}(g)=f(g\cdot x)$. Because the map $g\mapsto g\cdot x$ is always right-uniformly continuous (see \cite[Lemma 2.15]{Pe}), $f_{x}$ is always in $\RUCB(G)$, and one can check that $\{ f_{x} : f\in \mathrm C(X)\}$ is a unital left-invariant closed C$^{*}$-subalgebra of $\RUCB(G)$.  

Conversely, to every unital left-invariant closed C$^{*}$-subalgebra $\A$ of $\RUCB(G)$, one can associate a compact space $G^{\A}$, the \emph{Gelfand space} of $\A$. More details on this classical object will be given in Section \ref{section:ambits}. For the moment, we will just need that this is a compactification of $G$, on which the left-regular action by $G$ on itself naturally extends in a continuous way, and turns $G^{\A}$ into a $G$-flow. Furthermore, considering the point $e_{G}\in G^{\A}$, the map $\mathrm C(G^{\A})\to\A$ defined by  $f\mapsto f_{e_{G}}$ (as in the previous paragraph) realizes an isomorphism of C$^{*}$-algebras. This justifies the identification of $\mathrm C(G^{\A})$ with $\A$. In the sequel, we will use that fact under the following form: the entourages of the uniformity induced on $G$ by the compactification $G\to G^{\A}$ are of the form $\{ (g,h) \in G^{2} : \forall f\in \mathcal F \ \ |f(g)-f(h)|<\e\}$, where $\mathcal F\subset \A$ is finite, and $\e >0$.

\begin{defn}
Let $G$ be a topological group and $\mathcal F \subset \RUCB(G)$. Say that $\mathcal F$ is \emph{finitely oscillation stable} when for every finite $H\subset G$, $\varepsilon >0$, there exists $g\in G$ so that every $f\in \mathcal F$ is constant on $Hg$ up to $\varepsilon$: $$ \forall \e >0 \ \ \exists  g\in G \ \ \forall f\in \mathcal F \ \ \forall h,h' \in H \quad |f(hg)-f(h'g)|<\e$$

\end{defn}

This crucial notion is due to Pestov (for more on this, see \cite{Pe}), even if it was originally stated for left-uniformly continuous fonctions. The reason to deal with right-uniformly continuous fonctions here is that these are the ones that are naturally used to compactify $G$ in a way that is compatible with the left-regular action.  

\begin{prop}
Let $G$ be a topological group, $G\actson X$ a $G$-flow, and $x\in X$. TFAE: 
\begin{enumerate}
\item[i)] The orbit closure $\overline{G\cdot x}$ contains a fixed point. 
\item[ii)] For every finite $\mathcal F\subset C(X)$, the family $\{ f_{x}:f\in \mathcal F\}$ is finitely oscillation stable. 
\end{enumerate}

\end{prop}

\begin{proof}

$i)\Rightarrow ii)$: Fix $\mathcal F\subset C(X)$ finite, $H\subset G$ finite, $\varepsilon>0$. Let $y\in \overline{G\cdot x}$ be a fixed point. Thanks to the continuity of the elements of $\mathcal F$, we may find $g\cdot x$ close enough to $y$ so that for every $f\in \mathcal F$ and every $h\in H$, $|f(h\cdot g\cdot x) - f(h\cdot y)|<\e/2$, i.e. $|f(h\cdot g\cdot x) - f(y)|<\e/2$ since $y$ is fixed. Then, for every $f\in \mathcal F$, $h, h'\in H$, we have: \begin{align*}\ensuremath
\left|f_{x}(hg) - f_{x}(h'g)\right| & =  |f(h\cdot g\cdot x) - f(h'\cdot g\cdot x)| \\
& = |f(h\cdot g\cdot x) - f(y)| + |f(y) - f(h'\cdot g\cdot x)| \\
& <  \e
\end{align*}

$ii)\Rightarrow i)$: For $\mathcal F\subset C(X)$ finite, $H\subset G$ finite, $\varepsilon>0$, define $$ A_{\mathcal F, H, \varepsilon} = \{ y\in \overline{G\cdot x} : \forall f \in \mathcal F \ \ \forall h\in H \ \ |f(h\cdot y) - f(y)|\leq \varepsilon\}$$

This defines a family of closed subsets of $\overline{G\cdot x}$. Thanks to $ii)$, it has the finite intersection property (every finite intersection of its members contains an element of $G\cdot x$). Its intersection is therefore non-empty. Notice now that this intersection consists of fixed points. \end{proof}

As direct consequences, we obtain: 

\begin{prop}
\label{cor:fp}
Let $G$ be a topological group. Let $\A$ be a unital left invariant, closed C$^{*}$-subalgebra of $\RUCB(G)$. TFAE: 
\begin{enumerate}
\item[i)] The flow $G\actson G^{\A}$ has a fixed point. 
\item[ii)] Every finite $\mathcal F\subset \A$ is finitely oscillation stable. 
\end{enumerate}
\end{prop}

\begin{prop}
\label{cor:trivialmin}
Let $G$ be a topological group. Let $\A$ be a unital left-invariant, closed C$^{*}$-subalgebra of $\RUCB(G)$. TFAE: 
\begin{enumerate}
\item[i)] Every minimal subflow of $G\actson G^{\A}$ is trivial. 
\item[ii)] For every $x\in G^{\A}$, $\mathcal F\subset \A$ finite, the family $\{ f_{x}:f\in \mathcal F\}$ is finitely oscillation stable.
\end{enumerate}
\end{prop}

In the sequel, we will write $\mathcal F_{x}$ for the family $\{ f_{x}:f\in \mathcal F\}$. Note that the inclusion $\A\subset \bigcup\{\A_{x}:x\in G^{\A}\}$ may be strict. This is for example the case for the Roelcke algebra $\ROB(G)$ defined in Section \ref{section:Roelcke} (see \cite{GlM2}*{Corollary 4.11}). However, there are interesting cases where equality holds, e.g. $\RUCB(G)$ itself, the algebra $\WAP(G)$ of weakly almost periodic functions on $G$ (for a definition, see Section \ref{section:map}), or any of its closed left-invariant subalgebras \cite{BJM}*{Chapter III, Lemma 8.8}. A more detailed discussion about this topic and its dynamical interpretation in terms of point-universality can be found in \cite{GlM1} (Definition 2.5 and related material), \cite{GlM2}*{Sections 3 and 4} and \cite{GlM3}*{Remarks 4.15 and 4.16} by Glasner-Megrelishvili.

\section{Ramsey properties as natural combinatorial counterparts to the existence of fixed points}

\label{section:RPfp}

The purpose of this section is to show that when $G$ is of the form $\Aut(\m F)$ for some Fra\"iss\'e structure $\m F$, the existence of fixed points expressed in Proposition \ref{cor:fp} and Proposition \ref{cor:trivialmin} naturally translates combinatorially as Ramsey-theoretical statements. Precisely, our aim here is first to prove Theorem \ref{thm:fpRPcpct} below, and then Theorem \ref{thm:fpRP} (for definitions, see Sections \ref{section:colorings} and \ref{section:ambits}): 

\begin{thm}
\label{thm:fpRPcpct}
Let $\m F$ be a Fra\"iss\'e structure, and let $\A$ be a unital, left-invariant, closed C$^{*}$-subalgebra of $\RUCB(\Aut(\m F))$. TFAE: 
\begin{enumerate}
\item[i)] The flow $\Aut(\m F)\actson \Aut(\m F)^{\A}$ has a fixed point. 
\item[ii)] For every $\e>0$, $\m F$ has the Ramsey property up to $2\e$ for the finite colorings in $(\A)_{\e}$. 
\end{enumerate}

Those imply the following equivalent statements: 
\begin{enumerate}
\item[iii)] Every zero-dimensional factor of $\Aut(\m F)\actson \Aut(\m F)^{\A}$ has a fixed point. 
\item[iv)] $\m F$ has the Ramsey property for the finite colorings in $\A$. 
\end{enumerate}

When the finite colorings are dense in $\A$, all those statements are equivalent. 
\end{thm}



Even though Theorem \ref{thm:fpRP} and Theorem \ref{thm:fpRPcpct} look quite similar, we will see in the following sections that both of them will be handy when dealing with practical situations. This will lead to Theorem \ref{thm:MAPStableRP}, Theorem \ref{thm:fpRoecke} and Theorem \ref{thm:minRoeckeDefRP}. However, other natural algebras do not seem to admit approximations by finite colorings. We will see two such examples later on, with the proximal and the strongly proximal algebras.  

\subsection{Finite oscillation stability and Ramsey properties}

\label{section:colorings}

Let $\m F$ be a Fra\"iss\'e structure, whose underlying set is $\N$. As before, for a finite substructure $\m A \subset \m F$, let $\Stab(\m A) \subset \Aut(\m F)$ denote the pointwise stabilizer of $\m A$. Given any $g\in G$, its equivalence class $\bar g$ in the right quotient $\Stab(\m A)\backslash \Aut(\m F)$ is the set of all those elements of $G$ that are $\m A$-close to $g$ (i.e. some sort of ball of radius $\m A$) relative to the right uniformity, and can be thought as the restriction $\restrict{g^{-1}}{A}$, an embedding of $\m A$ into $\m F$. 

Furthermore, because $\m F$ is ultrahomogeneous, every element of $\binom{\m F}{\m A}$ is of that form. In other words, we can identify $\Stab(\m A)\backslash \Aut(\m F)$ and $\binom{\m F}{\m A}$. In addition, since every element of $\Stab(\m A)\backslash \Aut(\m F)$ can be thought as a ball for the right uniformity, every coloring of $\binom{\m F}{\m A}$, that is, every map $\bar\chi:\binom{\m F}{\m A}\rightarrow \C$, can be seen as an element $\chi$ of $\RUCB(\Aut(\m F))$ that is constant on small enough balls and satisfies $\chi(g)=\bar\chi(\bar g)$. In the sequel, we will usually not make any notational distinction between $\chi$ and $\bar\chi$, and by a \emph{coloring in} (resp. \emph{finite coloring in}) $\RUCB(\Aut(\m F))$, we will mean exactly a function $\chi$ of that kind (resp. with finite range). From this point of view, note that even if we allow $\m A$ to range over the set of all finite substructures of $\m F$, every finite set $\mathcal C \subset \RUCB(\Aut(\m F))$ of finite colorings can be seen as a finite set of finite colorings defined on the same set $\binom{\m F}{\m A}$, with values in a common set. 




\begin{defn}
Let $\mathcal F \subset \RUCB(\Aut(\m F))$ and $\e>0$. Say that $\m F$ has the \emph{Ramsey property (resp. Ramsey property up to $\e$) for colorings in $\mathcal F$} when for every $\m A, \m B$ in $\Age(\m F)$, every finite set $\mathcal C \subset \mathcal F$ of finite colorings of $\binom{\m F}{\m A}$, there exists $b\in \binom{\m F}{\m B}$ such that every $\chi\in \mathcal C$ is constant (resp. constant up to $\e$) on $\binom{b(\m B)}{\m A}$. 
\end{defn}

Note that as it is defined, the Ramsey property for colorings in $\mathcal F$ is a property of $\m F$, as opposed to a property of $\Age(\m F)$. We will meet several instances where it completely finitizes (e.g. Theorem \ref{thm:MAPStableRP}, Theorem \ref{thm:fpRoecke} and Theorem \ref{thm:minRoeckeDefRP}), but for the moment, this is only feasible via a case-by-case analysis.

\begin{prop}
\label{prop:aOSRamsey}
Let $\m A \in \Age(\m F)$, $\mathcal C$ be a finite set of finite colorings of $\binom{\m F}{\m A}$, $\e>0$. TFAE: 
\begin{enumerate}
\item[i)] For every finite $H\subset G$, there exists $g\in G$ so that every $\chi \in \mathcal C$ is constant up to $\varepsilon$ on $Hg$. 

\item[ii)] For every $\m B \in \Age(\m F)$, there exists $b\in \binom{\m F}{\m B}$ such that every $\chi\in \mathcal C$ is constant up to $\e$ on $\binom{b(\m B)}{\m A}$. 
\end{enumerate} 
\end{prop} 

\begin{proof}

The proof hinges on the following observation: Let $\m A$ be a finite substructure of $\m F$ and $H$ be a finite subset of $\Aut(\m F)$. For $h\in \Aut(\m F)$, recall that $\bar h$ denotes the equivalence class of $h$ in the quotient $\Stab(\m A)\backslash \Aut(\m F)$. As we have seen, $\bar h$ can be thought as the restriction $\restrict{h^{-1}}{A}$, so $\overline H:=\{ \bar h:h\in H\}$ can be seen as a finite set of embeddings of $\m A$ into $\m F$. As such, it is contained in some set of the form $\binom{\m B}{\m A}$ for some finite substructure $\m B$ of $\m F$. Next, if $g$ is fixed in $G$, we have: $$\overline{Hg} = \{ \overline{hg}:h\in H\}=\{ \restrict{g^{-1}\circ h^{-1}}{A} : h\in H\} \subset \binom{g^{-1}(\m B)}{\m A}$$

Conversely, if $\m B \subset \m F$ is a finite substructure, then there is $H\subset \Aut(\m F)$ finite so that $\binom{\m B}{\m A}\subset \overline H$, and if $g\in \Aut(\m F)$, then $\binom{g^{-1}(\m B)}{\m A} \subset \overline{Hg}$.  

We now go on with the proof. Assume that for every finite $H\subset G$, there exists $g\in G$ so that every $\chi \in \mathcal C$ is constant up to $\varepsilon$ on $Hg$. Let $H\subset \Aut(\m F)$ be a finite set so that $\binom{\m B}{\m A} \subset \overline H$. Find $g\in \Aut(\m F)$ such that every $\chi\in \mathcal C$ is constant up to $\varepsilon$ on $Hg$. Then every $\chi\in \mathcal C$ is constant up to $\e$ on $Hg$, and hence on $\overline{Hg}\supset \binom{g^{-1}(\m B)}{\m A}$. Therefore, it suffices to set $b=\restrict{g^{-1}}{B}$. 

Conversely, fix $H\subset \Aut(\m F)$ finite. Let $\m B$ be a finite substructure of $\m F$ so that $\overline{H} \subset \binom{\m B}{\m A}$. By hypothesis, find $b\in \binom{\m F}{\m B}$ such that every $\chi\in \mathcal C$ is constant up to $\e$ on $\binom{b(\m B)}{\m A}$. Take $g\in \Aut(\m F)$ such that $g^{-1}$ extends $b$. Then $Hg\subset \overline{Hg}\subset \binom{g^{-1}(\m B)}{\m A}$ and every $\chi\in \mathcal C$ is constant up to $\varepsilon$ on $Hg$.  
\end{proof}

\begin{prop}
\label{prop:OSRP}
Let $\m F$ be a Fra\"iss\'e structure, and let $\A$ be a unital, left-invariant, closed C$^{*}$-subalgebra of $\RUCB(\Aut(\m F))$. TFAE: 
\begin{enumerate}
\item[i)] Every finite $\mathcal F\subset \A$ is finitely oscillation stable.
\item[ii)] For every $\e>0$, $\m F$ has the Ramsey property up to $2\e$ for colorings in $(\A)_{\e}$. 
\end{enumerate}
\end{prop}

\begin{proof}

Assume that every finite $\mathcal F\subset \A$ is finitely oscillation stable. Fix $\m A$ in $\Age(\m F)$, $\mathcal C \subset (\A)_{\e}$ a finite set of finite colorings of $\binom{\m F}{\m A}$, $H\subset \Aut(\m F)$ finite. Fix $\{ f_{\chi}: \chi\in \mathcal C\}\subset \A$ and $\eta >0$ so that $\| \chi - f_{\chi}\|_{\infty}+\eta<\e$ for every $\chi \in \mathcal C$. By finite oscillation stability of $\{ f_{\chi}: \chi\in \mathcal C\}$, find $g\in \Aut(\m F)$ so that every $f_{\chi}$ is constant up to $\eta$ on $Hg$. Then every $\chi \in \mathcal C$ is constant up to $2\e$ on $Hg$. Thanks to Proposition \ref{prop:aOSRamsey}, we deduce that for every $\m B \in \Age(\m F)$, there exists $b\in \binom{\m F}{\m B}$ such that every $\chi\in \mathcal C$ is constant up to $2\e$ on $\binom{b(\m B)}{\m A}$. This is exactly what we needed to prove. 

Conversely, assume that $ii)$ holds, and fix $\mathcal F\subset \mathcal A$ finite, $\e>0$, $H\subset \Aut(\m F)$ finite. Let $\{ \chi_{f} : f\in \mathcal F\}$ be a finite family of finite colorings in $(\A)_{\e/4}$ so that $\| f-\chi_{f}\|_{\infty}<\e/4$ for every $f\in \mathcal F$. Thanks to Proposition \ref{prop:aOSRamsey}, $ii)$ implies that there is $g\in \Aut(\m F)$ so that every $\chi_{f}$ is constant up to $\e/2$ on $Hg$. Then, every $f\in \mathcal F$ is constant up to $\e$ on $Hg$ and $\mathcal F$ is finitely oscillation stable. 
\end{proof}

\subsection{Ramsey properties and fixed point in compactifications}

In this section, we prove Theorem \ref{thm:fpRPcpct}. Tying up Proposition \ref{prop:OSRP} with Proposition \ref{cor:fp}, we obtain: 

\begin{prop}
\label{cor:fpaRP}
Let $\m F$ be a Fra\"iss\'e structure, and let $\A$ be a unital, left-invariant, closed C$^{*}$-subalgebra of $\RUCB(\Aut(\m F))$. TFAE: 
\begin{enumerate}
\item[i)] The flow $\Aut(\m F)\actson \Aut(\m F)^{\A}$ has a fixed point. 
\item[ii)] For every $\e>0$, $\m F$ has the Ramsey property up to $2\e$ for colorings in $(\A)_{\e}$. 
\end{enumerate}
\end{prop}

%
%
%
%

Note the presence of the error term $2\e$ in item $ii)$ of the previous equivalence. Its appearance seems necessary in full generality, but can be removed under the additional assumption that finite colorings are dense in $\A$. In order to see this, observe first that considering all $\e>0$ simultaneously in Proposition \ref{prop:aOSRamsey}, one easily obtains:  

\begin{prop}
\label{prop:OSRamsey}
Let $\m A \in \Age(\m F)$, $\mathcal C$ be a finite set of finite colorings of $\binom{\m F}{\m A}$. TFAE: 
\begin{enumerate}
\item[i)] $\mathcal C$ is finitely oscillation stable. 
\item[ii)] For every $\m B \in \Age(\m F)$, there exists $b\in \binom{\m F}{\m B}$ such that every $\chi\in \mathcal C$ is constant on $\binom{b(\m B)}{\m A}$. 
\end{enumerate} 
\end{prop}

This yields: 

\begin{prop}
\label{cor:fpRP}
Let $\m F$ be a Fra\"iss\'e structure, and let $\A$ be a unital, left-invariant, closed C$^{*}$-subalgebra of $\RUCB(\Aut(\m F))$. Assume that finite colorings are dense in $\A$. TFAE: 
\begin{enumerate}
\item[i)] The flow $\Aut(\m F)\actson \Aut(\m F)^{\A}$ has a fixed point. 
\item[ii)] The structure $\m F$ has the Ramsey property for colorings in $\A$. 
\end{enumerate}
\end{prop}

\begin{proof}
Thanks to Proposition \ref{cor:fp}, the flow $\Aut(\m F)\actson \Aut(\m F)^{\A}$ has a fixed point iff every finite $\mathcal F\subset \A$ is finitely oscillation stable. Because finite colorings are dense in $\A$, this holds iff every finite set $\mathcal C \subset \A$ of finite colorings is oscillation stable. This is equivalent to $\m F$ having the Ramsey property for colorings in $\A$ thanks to Proposition \ref{prop:OSRamsey}. \end{proof}


\begin{proof}[Proof of Theorem \ref{thm:fpRPcpct}]

The equivalence $i)\Leftrightarrow ii)$ follows from Proposition \ref{cor:fpaRP}. For $iii)\Leftrightarrow iv)$, consider $\mathcal B$ the unital, left-invariant, closed C$^{*}$-subalgebra of $\A$ generated by the set of all finite colorings in $\A$. By Proposition \ref{cor:fpRP}, the flow $\Aut(\m F)\actson \Aut(\m F)^{\mathcal B}$ has a fixed point iff $\m F$ has the Ramsey property for colorings in $\mathcal B$, which is equivalent to the Ramsey property for colorings in $\A$. Therefore, it suffices to show that $\Aut(\m F)\actson \Aut(\m F)^{\mathcal B}$ has a fixed point iff every zero-dimensional factor of $\Aut(\m F)\actson \Aut(\m F)^{\mathcal A}$ does. To do this, recall that a compact topological space $X$ is zero-dimensional exactly when the continuous maps taking finitely many values are uniformly dense in $C(X)$. It follows that $\Aut(\m F)^{\mathcal B}$ is zero-dimensional, which proves one implication. For the other one, let $\Aut(\m F)\actson X$ be a zero-dimensional factor of $\Aut(\m F)\actson \Aut(\m F)^{\mathcal A}$, as witnessed by the map $\pi : \Aut(\m F)^{\A}\rightarrow X$. Let $x=\pi(e_{\Aut(\m F)})$. Then $C(X)_{x}\subset \A$. Since $X$ is zero-dimensional, the continuous maps taking finitely many values are dense in $C(X)$, so finite colorings are dense in $C(X)_{x}$. Therefore, we have in fact $C(X)_{x}\subset \mathcal B$ and by duality $(\Aut(\m F)^{C(X)_{x}},x)\cong (\overline{G\cdot x}, x)$ is a factor of $\Aut(\m F)\actson \Aut(\m F)^{\mathcal B}$. Since this latter flow has a fixed point, so does the former one. \end{proof}

The following result, which can be thought as a combinatorial counterpart to Proposition \ref{cor:trivialmin}, is an easy corollary: 

\begin{cor}
\label{cor:fpRP'}
Let $\m F$ be a Fra\"iss\'e structure, and let $\A$ be a unital, left-invariant, closed C$^{*}$-subalgebra of $\RUCB(\Aut(\m F))$. TFAE: 
\begin{enumerate}
\item[i)] Every minimal subflow of the flow $\Aut(\m F)\actson \Aut(\m F)^{\A}$ is trivial. 
\item[ii)] For every $x\in \Aut(\m F)^{\A}$, $\e>0$, the structure $\m F$ has the Ramsey property up to $2\e$ for colorings in $(\A_{x})_{\e}$. 
\end{enumerate}

Those imply the following equivalent statements: 
\begin{enumerate}
\item[iii)] Every minimal zero-dimensional subflow of $\Aut(\m F)\actson \Aut(\m F)^{\A}$ is trivial. 
\item[iv)] For every $x\in \Aut(\m F)^{\A}$, the structure $\m F$ has the Ramsey property for colorings in $\A_{x}$. 
\end{enumerate}

When finite colorings are dense in $\A$, all those statements are equivalent.  
\end{cor}

%
%
%
%
%
%
%

\subsection{Ramsey properties and fixed points in classes of flows}

\label{section:ambits}

In this section, we prove Theorem \ref{thm:fpRP}. We have just seen how Ramsey-theoretical statements reflect the existence of fixed points in certain compactifications. In practice, however, one is often interested in the existence of fixed points in a given class $\mathcal X$ of flows defined by a dynamical property (like being distal, equicontinuous, proximal,...), as opposed to the existence of a fixed point in a particular compactification. The purpose of what follows is to show that in that setting, the Ramsey-theoretical approach remains relevant at the cost of rather mild hypotheses on $\mathcal X$. The reader familiar with topological dynamics and Gelfand compactifications may go directly to the proof Theorem \ref{thm:fpRP}, at the end of this section. For the other ones, a synthetic treatment based on \cite[Chapter IV, Sections 4 and 5]{dV} is presented below. This material is classical and is only included here for the sake of completeness.

In what follows, it will be convenient to work with $\mathcal X$-$G$-ambits, i.e. $G$-ambits $G\actson(X,x)$ so that $G\actson X \in \mathcal X$. Recall first that for a family $(X_{\alpha}, x_{\alpha})_{\alpha}$ of $G$-ambits, its \emph{supremum} $\bigvee_{\alpha}(X_{\alpha},x_{\alpha})$ is the $G$-ambit induced on the orbit closure of $(x_{\alpha})_{\alpha}$ in the product $\prod_{\alpha}X_{\alpha}$, together with the distinguished point $(x_{\alpha})_{\alpha}$. Next, consider the algebra $\RUCB(G)$. We have already seen that $G$ acts continuously on it by left-shift via $g\cdot f (h)=f(g^{-1}h)$. It also acts by right shift via $g\bullet f(h):=f(hg)$. It turns out that when $\RUCB(G)$ is equipped with the \emph{pointwise convergence} topology, this action is continuous\footnote{Caution: Continuity may not hold on $\RUCB(G)$ itself. I am grateful to the referee for having pointed it out.} on the orbit (pointwise) closure $\overline{G\bullet f}$ of every $f\in \RUCB(G)$. This set is then a compact invariant subset of $\RUCB(G)$, to which one can attach the $G$-ambit $(\overline{G\bullet f}, f)$. The reason for which this ambit is relevant here comes from the following fact: 

\begin{prop}
\label{prop:Gelfand}
Let $G$ be a topological group, $f\in \RUCB(G)$. Let $\langle f \rangle$ denote the unital left-invariant, closed C$^{*}$-subalgebra of $\RUCB(G)$ generated by $f$. Then the ambits $(G^{\langle f \rangle},e_{G})$ and $(\overline{G\bullet f},f)$ are isomorphic. 
\end{prop}  

To prove this proposition, we start by making more explicit the construction of Gelfand compactifications. Let $\A$ be a unital left-invariant, closed C$^{*}$-subalgebra of $\RUCB(G)$. The Gelfand space $G^{\A}$ is, by definition, the space of C$^{*}$-algebra homomorphisms $\phi:\A\rightarrow \C$. It is compact when equipped with its weak$^{*}$-topology. Every $g\in G$ defines an evaluation fonctional $\hat g:\alpha\mapsto\alpha(g)$, and this defines a compactification of $G$, on which the left-regular action of $G$ on itself extends naturally to an action on $G^{\A}$ by left-shift $g\cdot \phi(\alpha)=\phi(g^{-1}\cdot \alpha)$. Here are the crucial features of $G^{\A}$ that we will use: 
\begin{enumerate}
\item $\mathrm C(G^{\A})$ can be identified with $\A$. This is realized by the isomorphism of C$^{*}$-algebras $\mathrm C(G^{\A})\to\A$ defined by  $f\mapsto f_{e_{G}}$, and whose inverse sends $\alpha\in \A$ to the continuous function $\hat \alpha$ defined on $G^{\A}$ by $\hat \alpha : \phi\mapsto \phi(\alpha)$. 
\item Duality: If $\A, \mathcal B$ are two unital left-invariant, closed C$^{*}$-subalgebra of $\RUCB(G)$, then $\A \subset \mathcal B$ holds iff $(G^{\A}, e_{G})$ is a factor of $(G^{\mathcal B}, e_{G})$. 
\item Let $G\actson(X,x)$ be a $G$-ambit. Then the unital left-invariant, closed C$^{*}$-subalgebra $C(X)_{x}$ of $\RUCB(G)$ defined by $C(X)_{x} = \{ f_{x}:f\in C(X)\}$ (recall that $f_{x}(g)=f(g\cdot x)$) is such that $(G^{C(X)_{x}},e_{G})$ is isomorphic to $(X,x)$. 
\end{enumerate}

With all this in mind, let us now turn to the proof of Proposition \ref{prop:Gelfand}. 

\begin{proof}[Proof of Proposition \ref{prop:Gelfand}]
As we have seen in Section \ref{section:oscillationstability}, $f$ can be thought as the continuous function $\hat f$ on $G^{\langle f \rangle}$ defined by $\hat f:\phi\mapsto \phi(f)$. It follows that for every $\phi\in G^{\langle f \rangle}$ the map $\pi(\phi):h\mapsto \hat f (h\cdot \phi) \ (=h\cdot \phi(f) = \phi(h^{-1}\cdot f))$ is in $\RUCB(G)$. This defines $$\pi:G^{\langle f \rangle}\rightarrow \RUCB(G)$$

Note that for $g,h\in G$, $\pi(\hat g)(h)=\hat g(h^{-1}\cdot f)=(h^{-1}\cdot f)(g)=f(hg)=g\bullet f(h)$. Therefore, $\pi(\hat g)=g\bullet f$ and in particular $\pi(e_{G})=f$. Let us now verify that $\pi$ is an injective homomorphism of $G$-flows. This will suffice to prove the desired result, since $\pi$ will then be a $G$-flow isomorphism between $G^{\langle f \rangle}$ and its image in $ \RUCB(G)$, which is $\overline{G\bullet \pi(e_{G})}=\overline{G\bullet f}$. 

For injectivity, assume that $\pi(\phi_{1})=\pi(\phi_{2})$. From the expression of $\pi(\phi)(h)$ above, this implies that $\phi_{1}$ and $\phi_{2}$ agree on the orbit $G\cdot f$, and therefore on all of $\langle f \rangle$.  To prove that $\pi$ is $G$-equivariant, consider $g,h\in G$ and $\phi\in G^{\langle f \rangle}$. Then: $$ \pi(g\cdot \phi)(h)=(g\cdot \phi)(h^{-1}\cdot f)=\phi(g^{-1}\cdot(h^{-1}\cdot f))=\phi((hg)^{-1}\cdot f)=\pi(\phi)(hg)$$ 

The last term of the equality is $(g\bullet\pi(\phi))(h)$, so $\pi(g\cdot \phi)=g\bullet\pi(\phi)$. To prove that $\pi$ is continuous, fix $H\subset G$ finite, $\e>0$. If $\phi_{1}, \phi_{2}\in G^{\langle f \rangle}$ agree on the finite set $H^{-1}\cdot f$, then $|\phi_{1}(h^{-1}\cdot f)-\phi_{2}(h^{-1}\cdot f)|<\e$ for every $h\in H$. This means that for every $h\in H$, $|\pi(\phi_{1})(h)-\pi(\phi_{2})(h)|<\e$, as required. \end{proof}

Before going on, a small remark: We now have two actions of $G$ on $\RUCB(G)$. When $G=\Aut(\m F)$ for a Fra\"iss\'e structure $\m F$, we have seen the set of finite colorings as a subset of $\RUCB(\Aut(\m F))$, consisting of those functions $\chi$ such that $\chi(h)=\bar{\chi}({\restrict{h^{-1}}{A}})$ for some finite $\m A$. Thus, $$g\bullet \chi (h) = \chi(hg)=\bar\chi(\restrict{g^{-1}h^{-1}}{A})=g\cdot\bar\chi(\restrict{h^{-1}}{A})$$ 

In other words, the action by right shift on $\RUCB(\Aut(\m F))$ induces the action by left shift on the space of finite colorings. In counterpart, the action by left shift on $\RUCB(\Aut(\m F))$ does not seem to transfer naturally to the space of colorings.   

\begin{prop}
\label{prop:KA}
Let $G$ be a topological group and $\mathcal X$ be a class of $G$-flows such that the class of $\mathcal X$-$G$-ambits is closed under suprema and factors. Then the set $\A=\{ f\in \RUCB(G) : G\actson \overline{G\bullet f} \in \mathcal X\}$ forms a unital left-invariant, closed C$^{*}$-subalgebra of $\RUCB(G)$, and the factors of $G\actson(G^{\A}, e_{G})$ are exactly the $\mathcal X$-$G$-ambits. 
\end{prop}

\begin{proof}
Let $(X,x)=\bigvee_{f\in \A}(\overline{G\bullet f},f)$. As a supremum of $\mathcal X$-$G$-ambits, it is a $\mathcal X$-$G$-ambit as well. Let $C(X)_{x} = \{ f_{x}:f\in C(X)\}$. As we have seen, this is a unital left-invariant, closed C$^{*}$-subalgebra of $\RUCB(G)$. To prove the result, it suffices to show that it is equal to $\A$. Let $f\in \RUCB(G)$. Then $f\in C(X)_{x}$ iff $\langle f\rangle \subset C(X)_{x}$. Passing to Gelfand compactifications, this means that $(G^{\langle f \rangle},e_{G})$ is a factor of $(G^{C(X)_{x}},e_{G})$, or, equivalently, that 
$(\overline{G\bullet f},f)$ is a factor of $(X,x)$ (Proposition \ref{prop:Gelfand}). Now, this happens iff $f\in \A$: the direct implication holds because the class of $\mathcal X$-$G$-ambits is closed under factors, and the converse holds thanks to the definition of $(X,x)$, as $(\overline{G\bullet f},f)$ appears as one of its factors. \end{proof}

\begin{proof}[Proof of Theorem \ref{thm:fpRP}]
In view of the previous proposition, it follows at once that $G\actson G^{\A}$ (resp. every zero-dimensional factor of $G\actson G^{\A}$) has a fixed point iff every $\mathcal X$-$G$-ambit (resp. zero-dimensional $\mathcal X$-$G$-ambit) has a fixed point. When $\mathcal X$ satisfies the additional property that every $G\actson X\in \mathcal X$ admits some $x\in X$ such that $G\actson \overline{G\cdot x}\in \mathcal X$, those statements are equivalent to the fact that every $G$-flow (resp. zero-dimensional $G$-flow) in $\mathcal X$ has a fixed point. Theorem \ref{thm:fpRP} now follows from Theorem \ref{thm:fpRPcpct}.   
\end{proof}

\section{Roelcke flows and definable colorings}

\label{section:Roelcke}

The purpose of this section is to prove Theorem \ref{thm:fpRoecke} and Theorem \ref{thm:minRoeckeDefRP} thanks to the machinery that we just developed. This is done in Section \ref{subsection:Roelcke} and Section \ref{subsection:defRP}, respectively. We finish in Section \ref{subsection:Remarks} with several remarks.   


\subsection{Fixed points in the Roelcke compactification, Roelcke colorings and joint embedding patterns}

\label{subsection:Roelcke}

\begin{defn}
Let $f : G \rightarrow \C$. It is \emph{Roelcke} when it is uniformly continuous relative to the Roelcke uniformity on $G$.
\end{defn}

Equivalently, $f$ is Roelcke when it is both right and left uniformly continuous on $G$. In what follows, we will be particularly interested in \emph{Roelcke-precompact} groups, i.e. groups with compact completion relative to the Roelcke uniformity. In that case, every Roelcke function on $G$ is bounded, and the set $\ROB(G)$ of all Roelcke, bounded, functions is a unital, left-invariant, closed C$^{*}$-subalgebra of $\RUCB(G)$. The corresponding compactification $G^{\ROB(G)}$ will be denoted by $R(G)$. After their introduction in by Roelcke and Dierolf \cite{RD}, Roelcke-precompact groups have shown their utility through the work of Uspenskij \cites{Us3, Us4, Us5}. More recently, several essential contributions by Tsankov \cite{Ts}, Ben-Yaacov-Tsankov \cite{BYT} and Ibarluc\'ia \cite{I0, I1} have shown that their r\^{o}le is central when studying automorphism groups of Fra\"iss\'e structures from the model-theoretic point of view. As a matter of fact, Roelcke-precompact groups of the form $\Aut(\m F)$ for $\m F$ Fra\"iss\'e can be easily characterized combinatorially. Indeed, we have seen in Section \ref{section:oscillationstability} that a typical entourage of the Roelcke uniformity on $\Aut(\m F)$ is indexed by two finite substructures $\m A$ and $\m Z$ of $\m F$, and that two elements $g,h\in \Aut(\m F)$ are $(\m A, \m Z)$-close when $$\Stab(\m A)g\Stab(\m Z)=\Stab(\m A)h\Stab(\m Z)$$ 

If we denote by $z$ the identity embedding $\m Z\to \m F$, this means: $$\langle \restrict{g^{-1}}{A}, z \rangle\cong \langle \restrict{h^{-1}}{A}, z \rangle$$ 

For the sequel, it will be useful to remember that for a joint embedding $\langle a,z\rangle$, $[a,z]$ refers to its pattern, i.e. its isomorphism type.   

\begin{prop}
\label{prop:Roelcke-precompact}
Let $\m F$ be a Fra\"iss\'e structure. Then $\Aut(\m F)$ is Roelcke-precompact iff for every $\m A, \m Z \in \Age(\m F)$, there are only finitely many joint embedding patterns of $\m A$ and $\m Z$. 
\end{prop}

\begin{proof}
$\Aut(\m F)$ is Roelcke-precompact iff for every entourage $U$ there are $g_{1},...,g_{n}\in \Aut(\m F)$ so that every $g\in \Aut(\m F)$ is $U$-close to some $g_{i}$. From the discussion above, this means that for any two finite substructures $\m A, \m Z$ of $\m F$, $\Aut(\m F)$ can be covered by finitely many $\Stab(\m A)\backslash \Aut(\m F)/\Stab(\m Z)$-classes, which holds exactly when there are only finitely many joint embedding patterns of $\m A$ and $\m Z$. 
\end{proof}

\begin{prop}
Let $\m F$ be a Fra\"iss\'e structure. Then finite colorings are dense in $\ROB(\Aut(\m F))$. \end{prop}

\begin{proof}
Let $f\in \ROB(\Aut(\m F))$ and fix $\e>0$. Since $f$ is bounded, there is a finite set $Y$ so that the range of $f$ is contained in $(Y)_\e$. By uniform continuity of $f$, there are two finite substructures $\m A, \m Z$ of $\m F$ so that $f$ is constant up to $\e$ on every $\Stab(\m A)\backslash \Aut(\m F)/\Stab(\m Z)$-class. For any such class $P$, choose $ h_P\in P, y_P \in Y$  such that $|y_P - f(h_P)|<\e$. For $g\in G$, set $\bar{\bar g}:= \Stab(\m A)g\Stab(\m Z)$, the equivalence class of $g$ in $\Stab(\m A)\backslash \Aut(\m F)/\Stab(\m Z)$. Then, the map $\chi : g\mapsto y_{\bar{\bar g}}$ is a finite coloring of $\binom{\m F}{\m A}$. It is in $\ROB(\Aut(\m F))$ because it is constant on the $\Stab(\m A)\backslash \Aut(\m F)/\Stab(\m Z)$-classes, and in addition, for any $g\in G$: $$|\chi(g)-f(g)|=|y_{\bar{\bar g}} - f(g)|\leq |y_{\bar{\bar g}}-f(h_{\bar{\bar g}})|+ |f(h_{\bar{\bar g}})-f(g)|<2\e. \quad \qedhere$$
\end{proof}

Thanks to Theorem \ref{thm:fpRPcpct}, it follows that under the precompactness assumption of $\Aut(\m F)$, every Roelcke flow has a fixed point iff $\m F$ has the Ramsey property for finite colorings in $\ROB(\Aut(\m F))$. To see how this leads to Theorem \ref{thm:fpRoecke}, we now turn to a description of those colorings that are in $\ROB(\Aut(\m F))$. In fact, the previous proof already provides such a description. Indeed, if $f$ is assumed to be a finite coloring, then it has to be constant on every $\Stab(\m A)\backslash \Aut(\m F)/\Stab(\m Z)$-class for $\m A, \m Z$ large enough. This means exactly that $f$ can be seen as a finite coloring of the joint embedding patterns of $\m A$ and $\m Z$. Therefore, we have just proved:  

\begin{prop}
\label{prop:Roelcke coloring}
Let $\m F$ be a Fra\"iss\'e structure, $\m A\in \Age(\m F)$ and $\chi$ be a finite coloring of $\binom{\m F}{\m A}$. Then $\chi\in \ROB(\Aut(\m F))$ iff there is a finite substructure $\m Z$ of $\m F$ such that $\chi(a)$ depends only on $[a,z]$, where $z$ is the identity embedding. 
\end{prop}

\begin{prop}
\label{prop:RoelckeRP}
Let $\m F$ be a Fra\"iss\'e structure. Then $\m F$ has the Ramsey property for colorings in $\ROB(\Aut(\m F))$ iff for every $\m A, \m B, \m Z\in \Age(\m F)$, every $z\in \binom{\m F}{\m Z}$, every finite coloring $\gamma$ of the joint embedding patterns of $\m A$ and $\m Z$, there is $b\in \binom{\m F}{\m B}$ so that the coloring $a\mapsto \gamma([a,z])$ is constant on $\binom{b(\m B)}{\m A}$. 
\end{prop}

\begin{proof}
Assume that $\m F$ has the Ramsey property for colorings in $\ROB(\Aut(\m F))$, and fix $\m A, \m B, \m Z\in \Age(\m F)$, $z\in \binom{\m F}{\m Z}$, $\gamma$ a finite coloring of the joint embedding patterns of $\m A$ and $\m Z$. Then the coloring defined on $\binom{\m F}{\m A}$ by $a\mapsto \gamma([a,z])$ is in $\ROB(\Aut(\m F))$ by Proposition \ref{prop:Roelcke coloring}. The conclusion follows. The converse is an immediate consequence of the following easy fact: if $\mathcal C = \{ \chi_{i}:i\leq n\} \subset \ROB(\Aut(\m F))$ is a finite set of finite colorings, which we may assume to be colorings of $\binom{\m F}{\m A}$, Proposition \ref{prop:Roelcke coloring} guarantees that each $\chi_{i}$ is associated to some finite $\m Z^{i}\subset \Age(\m F)$ and finite coloring $\gamma_{i}$ of the joint embedding patterns of $\m A$ and $\m Z^{i}$. Then, we see that the hypothesis applied to $\m Z=\bigcup_{i\leq n}z^{i}(\m Z^{i})$ and $\gamma$ defined by $\gamma([a,z]) = (\gamma_{i}([a,z^{i}]))_{i\leq n}$ provides $b\in \binom{\m F}{\m B}$ so that for every $i\leq n$, the coloring $a\mapsto \gamma_{i}([a,z^{i}])$ is constant on $\binom{b(\m B)}{\m A}$. \end{proof}

\begin{prop}
\label{prop:RPRoelcke}
Let $\m F$ be a Fra\"iss\'e structure. Then $\m F$ has the Ramsey property for colorings in $\ROB(\Aut(\m F))$ iff for every $\m A, \m B, \m Z\in \Age(\m F)$, every finite coloring $\gamma$ of the joint embedding patterns of $\m A$ and $\m Z$, there exists a joint embedding $\langle b,z\rangle$ such that the coloring $a\mapsto \gamma([a,z])$ is constant on $\binom{b(\m B)}{\m A}$. 
\end{prop}

\begin{proof}
Assume that the Ramsey property for colorings in $\ROB(\Aut(\m F))$ holds, and fix $\m A, \m B, \m Z\in \Age(\m F)$, $\gamma$ a finite coloring of the joint embedding patterns of $\m A$ and $\m Z$. Fix $z\in \binom{\m F}{\m Z}$. Then $b\in \binom{\m F}{\m B}$ obtained by Proposition \ref{prop:RoelckeRP} is as required. Conversely, fix $\m A, \m B, \m Z\in \Age(\m F)$, $z\in \binom{\m F}{\m Z}$, $\gamma$ a finite coloring of the joint embedding patterns of $\m A$ and $\m Z$. Consider a joint embedding $\langle b',z'\rangle$ such that $\gamma([\cdot,z'])$ is constant on $\binom{b'(\m B)}{\m A}$. Let $i$ be the unique isomorphism such that $i\circ z'=z$. Then by ultrahomogeneity of $\m F$, we can extend $i$ to $b'(\m B)\cup z'(\m Z)$, and $b:=i\circ b'$ is as required. 
\end{proof}

\begin{proof}[Proof of Theorem \ref{thm:fpRoecke}]
Thanks to Theorem \ref{thm:fpRPcpct}, $\Aut(\m F)\actson R(\Aut(\m F))$ has a fixed point iff $\m F$ has the Ramsey property for colorings in $\ROB(\Aut(\m F))$. Apply then Proposition \ref{prop:RPRoelcke}. When $\Aut(\m F)$ is Roelcke-precompact, the additional statement is a reformulation of $ii)$ with the coloring $\gamma([a,z]):=[a,z]$, which is finite by Proposition \ref{prop:Roelcke-precompact}.  
\end{proof}

\subsection{Trivial minimal subflows in the Roelcke compactification and the definable Ramsey property}

\label{subsection:defRP}

We now turn to the proof of Theorem \ref{thm:minRoeckeDefRP}, where we do assume from the beginning that $\Aut(\m F)$ is Roelcke-precompact. Since we wish to do so via an application of Corollary \ref{cor:fpRP'}, we need to understand first how functions of the form $f_{x}$ look like when $f\in \ROB(\Aut(\m F))$ and $x\in R(\Aut(\m F))$. This is possible thanks to a convenient representation of the elements of $R(\Aut(\m F))$. Thanks to the discussion at the beginning of Subsection \ref{subsection:Roelcke}, a typical open neighborhood around a point $g\in \Aut(\m F)$ in $R(\Aut(\m F))$ is determined by all those $h\in \Aut(\m F)$ so that $\langle \restrict{h^{-1}}{A}, z \rangle \cong \langle \restrict{g^{-1}}{A}, z \rangle$, where $\m A$ and $\m Z$ are finite substructures of $\m F$ and $z$ is the natural inclusion map of $\m Z$ in $\m F$. In particular, letting $\m A$ and $\m Z$ being equal to the substructure $\m F_{n}$ of $\m F$ supported by $\{ k : k\leq n\}$ for each $n\in \N$ (recall that $\m F$ is based on $\N$), we obtain the nested sequence of clopen sets $$[\restrict{g^{-1}}{\m F_{n}}, \restrict{e_{\Aut(\m F)}}{\m F_{n}}]$$ whose intersection can be thought as $[g^{-1},e_{\Aut(\m F)}]$. In other words, in $R(\Aut(\m F))$, $g\in \Aut(\m F)$ is identified with $[g^{-1},e_{\Aut(\m F)}]$. In general, it is not too difficult to see that in $R(\Aut(\m F))$, a Cauchy sequence of elements of $\Aut(\m F)$ essentially corresponds to a coherent sequence of joint embedding patterns of two copies of $\m F_{0}$, $\m F_{1}$, $\m F_{2}$... which naturally converges to the pattern $[\phi_{1},\phi_{2}]$ of a joint embedding $\langle\phi_{1},\phi_{2}\rangle$ of two copies of $\m F$. A basic open neighborhood around this point is of the form $[\restrict{\phi_{1}}{\m A},\restrict{\phi_{2}}{\m Z}]$, with $\m A, \m Z$ finite substructures of $\m F$. To describe the action $\Aut(\m F)\actson R(\Aut(\m F))$, it suffices to observe that for $g, h\in \Aut(\m F)$, $gh$ is identified with $[h^{-1}\circ g^{-1},e_{\Aut(\m F)}]$. So, in general, since the action of $\Aut(\m F)$ on $R(\Aut(\m F))$ extends the left-regular action of $\Aut(\m F)$ on itself, we have,  for every $[\phi_{1},\phi_{2}]\in R(\Aut(\m F))$, $$g\cdot[\phi_{1},\phi_{2}] = [\phi_{1}\circ g^{-1},\phi_{2}]$$ 

\begin{prop}
Let $\m F$ be a Fra\"iss\'e structure with Roelcke-precompact automorphism group, and $x\in R(\Aut(\m F))$. Then $\ROB(\Aut(\m F))_{x}$ can be approximated by finite colorings. 
\end{prop}

\begin{proof}
Let $x\in R(\Aut(\m F))$, $f\in \ROB(\Aut(\m F))$ and $\e>0$. From the previous discussion, $x$ is of the form $x=[\phi_{1},\phi_{2}]$ and $f_{x}(g)=f([\phi_{1}\circ g^{-1},\phi_{2}])$ (where $f$ is now seen as a continuous function on $R(\Aut(\m F)$). By uniform continuity of $f$, there are two finite substructures $\m A, \m Z$ of $\m F$ so that for every $g,h\in \Aut(\m F)$, $$\langle\restrict{\phi_{1}\circ g^{-1}}{\m A}, \restrict{\phi_{2}}{\m Z}\rangle\cong\langle\restrict{\phi_{1}\circ h^{-1}}{\m A}, \restrict{\phi_{2}}{\m Z}\rangle \Rightarrow |f_{x}(g)-f_{x}(h)|<\e$$ 

Since $\Aut(\m F)$ is Roelcke-precompact, by Proposition \ref{prop:Roelcke-precompact}, there are only finitely joint embedding patterns of the form $[\phi_{1}\circ a, \phi_{2}\circ z]$. By choosing appropriate constants for each of these, we obtain $\chi\in\ROB(\Aut(\m F))_{x}$ so that $\|f_{x}-\chi\|_{\infty}<\e$, and which can be thought as a finite coloring of $\binom{\m F}{\m A}$. 
\end{proof}

As in Proposition \ref{prop:Roelcke coloring}, the previous proof also provides a description of those finite colorings that are in $\ROB(\Aut(\m F))_{x}$: if $f_{x}$ is assumed to be a finite coloring, then for $\m A, \m Z$ large enough finite substructures of $\m F$, it has to give same value to any two $g, h\in \Aut(\m F)$ which satisfy $\langle\restrict{\phi_{1}\circ g^{-1}}{\m A}, \restrict{\phi_{2}\circ z}{\m Z}\rangle\cong\langle\restrict{\phi_{1}\circ h^{-1}}{\m A}, \restrict{\phi_{2}\circ z}{\m Z}\rangle$. This means exactly that $f_{x}$ can be seen as a finite coloring of $\binom{\m F}{\m A}$ whose value at $a$ depends only on the joint embedding pattern $[\phi_{1}\circ a,\phi_{2}\circ z]$.
We have just proved:

\begin{prop}
\label{prop:Roelcke ext coloring}
Let $\m F$ be a Fra\"iss\'e structure with Roelcke-precompact automorphism group, $x=[\phi_{1},\phi_{2}]\in R(\Aut(\m F))$, $\m A\in \Age(\m F)$ and $\chi$ be a finite coloring of $\binom{\m F}{\m A}$. Then $\chi\in \ROB(\Aut(\m F))_{x}$ iff there is $\m Z\in \Age(\m F)$ and a joint embedding of the form $\langle \phi_{1}, \phi_{2}\circ z \rangle$ of $\m F$ and $\m Z$ such that $\chi(a)$ depends only on the joint embedding pattern $[\phi_{1}\circ a,\phi_{2}\circ z]$.  

\end{prop}

\begin{prop}
\label{prop:RPExtRoelcke}
Let $\m F$ be a Fra\"iss\'e structure with Roelcke-precompact automorphism group and let $x=[\phi_{1},\phi_{2}]\in R(\Aut(\m F))$. Then $\m F$ has the Ramsey property for colorings in $\ROB(\Aut(\m F))_{x}$ iff for every $\m A, \m B \in \Age(\m F)$, every $\m Z\in \Age(\m F)$, every joint embedding of the form $\langle \phi_{1},\phi_{2}\circ z\rangle$ of $\m F$ and $\m Z$, there is $b\in \binom{\m F}{\m B}$ so that the coloring $a\mapsto [ \phi_{1}\circ a,\phi_{2}\circ z]$ does not depend on $a$ on $\binom{b(\m B)}{\m A}$. 

\end{prop}

\begin{proof}
Assume that $\m F$ has the Ramsey property for colorings in $\ROB(\Aut(\m F))_{x}$, and fix $\m A, \m B, \m Z\in \Age(\m F)$ together with a joint embedding of $\m F$ and $\m Z$ of the form $\langle \phi_{1}, \phi_{2}\circ z\rangle$. Then the coloring defined on $\binom{\m F}{\m A}$ by $a\mapsto [\phi_{1} \circ a,\phi_{2}\circ z]$ is finite by Proposition \ref{prop:Roelcke-precompact}, and is in $\ROB(\Aut(\m F))_{x}$ by Proposition \ref{prop:Roelcke ext coloring}. The conclusion follows. The converse is an immediate consequence of Proposition \ref{prop:Roelcke ext coloring}, and of the fact that any finite set $\mathcal C$ of finite colorings in $\ROB(\Aut(\m F))_{x}$ can be captured by one single such coloring, as in the proof of Proposition \ref{prop:RoelckeRP}. 
\end{proof}

\begin{prop}
\label{prop:DefinableRP}
Let $\m F$ be a Fra\"iss\'e structure with Roelcke-precompact  automorphism group. Then $\m F$ has the Ramsey property for colorings in $\ROB(\Aut(\m F))_{x}$ for every $x\in R(\Aut(\m F))\}$ iff $\Age(\m F)$ has the definable Ramsey property. 
\end{prop}

\begin{proof}
From Proposition \ref{prop:RPExtRoelcke}, it appears that the definable Ramsey property is nothing else than a finitization of the fact that for every $x\in R(\Aut(\m F))$, the Ramsey property holds for colorings in $\ROB(\Aut(\m F))_{x}$. This proves the converse implication. 

The direct implication is obtained by a standard compactness argument: Assume that we can find $\m A, \m B, \m Z$ finite substructures of $\m F$ such that for every $\m C\in \Age(\m F)$, there exists a joint embedding  $\langle c,z \rangle$ such that no $b\in \binom{\m C}{\m B}$ satisfies that the map $a\mapsto [a,z]$ is constant on $\binom{b(\m B)}{\m A}$. Consider now the sequence $(\m F_{n})_{n\in \N}$ of initial segments of $\m F$ (recall that $\m F$ is based on $\N$ and that $\m F_{n}$ is the substructures of $\m F$ supported by $\{ k : k\leq n\}$). Each comes with some joint embedding pattern $[\phi_{n}, z_{n}]$ witnessing the failure of the definable Ramsey property. Note that we may assume that each $\phi_{n}$ is just the natural inclusion map from $\m F_{n}$ in $\m F$. Closing off this set of joint embedding patterns under initial segments of the first coordinate, we obtain a countable set whose elements are $[\phi_{m}, z_{n}]$, with $m\leq n$. Setting $[\phi_{m}, z_{n}] \leq [\phi_{p}, z_{q}]$ when $m\leq p$ and $[\phi_{m},z_{q}]=[\phi_{m},z_{n}]$, this becomes a countable tree, which is finitely branching since $\Aut(\m F)$ is Roelcke-precompact (Proposition \ref{prop:Roelcke-precompact}). By K\"onig's lemma, this tree contains an infinite branch, which can be seen as a joint embedding pattern $[\phi, z]$ of $\m F$ and $\m Z$. By construction, there is no $b\in \binom{\phi(\m F)}{\m B}$ such that $a\mapsto [a,z]$ is constant on $\binom{b(\m B)}{\m A}$. Therefore, the Ramsey property for colorings in $\ROB(\Aut(\m F))_{x}$ fails for any $x=[\phi_{1}, \phi_{2}]\in R(\Aut(\m F))$ satisfying $\phi_{1}=\phi$ and $\phi_{2}$ extending $z$ to $\m F$. \end{proof}

\begin{proof}[Proof of Theorem \ref{thm:minRoeckeDefRP}]
By Corollary \ref{cor:fpRP'}, the minimal subflows of $\Aut(\m F)\actson R(\Aut(\m F))$ are trivial iff for every $x\in R(\Aut(\m F))$, $\m F$ has the Ramsey property for colorings in $\ROB(\Aut(\m F))_{x}$. Apply then Proposition \ref{prop:DefinableRP}. \end{proof}

%
%
%
%
%
%
%

\subsection{Remarks}

\label{subsection:Remarks}

\subsubsection{Roelcke flows}

It is easy to see that the factors of $G\actson R(G)$ are exactly the $G$-flows $G\actson X$ such that for some $x\in X$, the map $G\rightarrow X$, $g\mapsto g\cdot x$ is both left and right uniformly continuous. Equivalently, there exists a right-action $\overline{G\cdot x}\curvearrowleft G$ commuting with the action $G\actson X$ such that $$ \forall g\in G\ \ g\cdot x = x \cdot g$$

Note that $g\mapsto g\cdot x$ is right uniformly continuous for any $G$-flow, so the definition of Roelcke flow really lies on the left uniform continuity of this map. Note also that a subflow of a Roelcke flow may not be Roelcke itself. For that reason, while it is easy to translate Theorem \ref{thm:fpRoecke} in terms of Roelcke flows (it characterizes when every Roelcke flow has a fixed point), the meaning of Theorem \ref{thm:minRoeckeDefRP} is much less clear.  

As the class of Roelcke flows does not seem to be of particular interest, let us simply mention that it is quite closely related to the class of \emph{strongly continuous flows} as defined by Glasner-Mergrelishvili in \cite{GlM2}, which is much better-behaved. However, in the case of Roelcke precompact groups, Ibarluc\'ia has shown in \cite{I1} that the corresponding subalgebra of $\RUCB(G)$ corresponds to the weakly almost periodic algebra (see Section \ref{subsection:WAP}). The study of the fixed point property on strongly continuous flows therefore reduces to that of equicontinuous and distal flows, which are treated in Section \ref{section:map}.  

%
%
%
%

\subsubsection{Minimal almost periodicity of the orthogonal group of $\ell^{2}$}

\label{subsubsection:mapell2}

It was mentioned in introduction that the orthogonal group $O(\ell^{2})$ of $\ell^{2}$ equipped with the strong operator topology can be shown to be minimally almost periodic thanks to Theorem \ref{thm:fpRoecke}. Here is the proof: consider the class of all finite metric spaces with distances in $\Q$ that embed isometrically in an affinely independent way in $\ell^{2}$. This is a Fra\"iss\'e class, for which it is easy to show via some elementary geometry that item $ii)$ of Theorem \ref{thm:fpRoecke} holds. The corresponding Fra\"iss\'e limit $\m H_{\Q}^{\mathrm{ind}}$ is a countable dense metric subspace of $\ell^{2}$ (see \cite[Chapter 1, Section 4.3]{NVT1}, from which the proof can be adapted easily), whose isometry group is therefore minimally almost periodic. This group embeds continuously and densely into $O(\ell^{2})$, which suffices to reach the desired conclusion. Again, much more is known about that object - its unitary representations have been completely classified by Kirillov in \cite{Ki}; furthermore, it is in fact extremely amenable by a result of Gromov and Milman \cite{GM} - but the present proof is, in comparison, rather simple.

\subsubsection{Ramsey-like and amalgamation properties}

The connection between Ramsey-like and amalgamation properties originates from the fundamental work of Ne\v set\v ril and R\"odl: on the one hand, any Ramsey class of finite ordered structures must have the amalgamation property \cite{Ne1}; on the other hand, the partite construction from \cite{NR2} and its descendants (arguably among the most powerful methods in structural Ramsey theory so far) are entirely based on amalgamation. Theorem \ref{thm:fpRoecke} and Theorem \ref{thm:minRoeckeDefRP} strengthen this link, by showing that amalgamation suffices to express combinatorial partition properties whose dynamical content (fixed point or trivial minimal components in the  Roelcke compactification) is actually quite close to that of the usual Ramsey property (extreme amenability, i.e. fixed point or trivial minimal components in the Samuel compactification). 

\subsubsection{Induction and the definable Ramsey property}

Unlike the usual Ramsey property, the definable Ramsey property is particularly well-adapted to a treatment by induction. This is particularly true when the underlying language is finite, as finitely many base cases suffice to show that it holds in general. More precisely: Given $\m A, \m B, \m C, \m Z$, write $\m C \rightarrow (\m B)^{\m A}_{\m Z}$ when for every joint embedding  $\langle c,z\rangle$, there is $b\in \binom{\m C}{\m B}$ so that on $\binom{b(\m B)}{\m A}$, the joint embedding pattern $[a,z]$ does not depend on $a$. Then, when the language is finite with maximum arity $k$, the definable Ramsey property holds for $\Age(\m F)$ as soon as for every $\m A, \m B, \m Z \in \Age(\m F)$ with $|\m A|+|\m Z|\leq k$ there exists $\m C$ in $\Age(\m F)$ such that $\m C \rightarrow (\m B)^{\m A}_{\m Z}$. For example, for binary structures, it suffices to consider $|\m A|=|\m Z|=1$, which is notoriously simpler than the general case where no restriction is placed on $|\m A|$.  

\subsubsection{$\omega$-categoricity versus finite language}

The definable Ramsey property is one of the first Ramsey-type phenomena where the distinction between $\omega$-categorical structures and structures in a finite language appears so explicitly. This certainly deserves to be noticed in view of the still open problem which consists in finding a well-behaved class of Fra\"iss\'e structures that admit a precompact expansion where the Ramsey property holds, see \cite{BPT}. Recall that by a result of Zucker \cite{Z2}, this problem is equivalent to that of finding a well-behaved class of non-Archimedean Polish groups whose universal minimal flow is metrizable, see also \cite{MNT} and \cite{BYMT}. I conjectured in \cite{NVT6} that Roelcke precompact groups do fall into that category. This was disproved by Evans in 2015 thanks to the use of an intricate model-theoretic construction originally due to Hrushovski, but the problem remains open for the automorphism groups coming from a Fra\"iss\'e structure in a finite language. (Evans' example is also at the center of the recent work \cite{EHN}.) With this in mind, it will be interesting to see to which extent techniques from model theory allow a better grasp on the combinatorial property exhibited in Theorem \ref{thm:fpRoecke} or on the definable Ramsey property.

\section{Equicontinuous and distal flows, definable equivalence relations and stable colorings}

\label{section:map}

In this section, we concentrate on minimal almost periodicity and on the proof of Theorem \ref{thm:MAPStableRP}. The first part, consisting of the equivalence between $i)$ and $ii)$, is carried out in Section \ref{subsection:mapdefeqrel}, where several known facts about equicontinuity and minimal periodicity are reminded. The second part is completed in Section \ref{subsection:WAP}, which deals with weakly almost periodic functions. 

\subsection{Minimal almost periodicity, almost periodic colorings and definable equivalence relations}

\label{subsection:mapdefeqrel}

Given a topological group $G$, the class of equicontinuous ambits is closed under suprema and factors \cite[Chapter IV, Section 2.27]{dV}. Since equicontinuity passes to subflows, Theorem \ref{thm:fpRP} applies to the class of equicontinuous flows. The corresponding C$^{*}$-subalgebra of $\RUCB(G)$ can be determined by using that the restriction of a $G$-flow $G\actson X$ is equicontinuous on the orbit closure $\overline{G\cdot x}$ iff $$\forall U_{\varepsilon}\in \mathcal U_{X} \ \exists U_{\eta}\in \mathcal U_{X} \ \forall x_{1},x_{2}\in G\cdot x \quad (x_{1},x_{2})\in U_{\eta} \Rightarrow \forall g\in G \ (g\cdot x_{1}, g\cdot x_{2})\in U_{\varepsilon}$$ and it is not difficult to verify that we recover the classical result according to which the corresponding C$^{*}$-subalgebra of $\RUCB(G)$ is the \emph{almost-periodic algebra} $\AP(G)$, the subalgebra of $\ROB(G)$ consisting of all those $f\in \RUCB(G)$ such that the orbit $G\bullet f$ is norm-precompact in $\RUCB(G)$ (equivalently, the orbit $G\cdot f$ is norm-precompact, see \cite[Chapter IV, Sections 5.30 and 6.15]{dV}). The corresponding compactification $G^{\AP(G)}$, usually denoted $B(G)$, is the \emph{Bohr compactification of} $G$, and is always a compact group \cite{dV}*{Appendix (D.12)}. 

In view of Theorem \ref{thm:fpRP}, we could try to provide a Ramsey-type characterization of minimal almost periodicity. However, the problem is of slightly different flavor here. Indeed, unlike what happens with many other classes of flows, having a fixed point in $G^{\AP(G)}$ simply means that $G^{\AP(G)}$ is trivial. Equivalently: every almost periodic function on $G$ is constant. Formulating Theorem \ref{thm:fpRP} would become rather awkward in that case, as it would just express that $\Aut(\m F)$ is almost periodic iff $\m F$ has the Ramsey property for some class of colorings... Which all turn out to be constant! Instead, the right approach to adopt here is to analyze which class of colorings we would be talking about. 

\begin{prop}
\label{prop:approxAP}
Let $\m F$ be a Fra\"iss\'e structure with Roelcke precompact automorphism group. Then, finite colorings are dense in $\mathrm{AP}(\Aut(\m F))$. 
\end{prop} 

\begin{proof}
This proof is largely inspired from the proof of \cite{BYT}*{Proposition 4.7}. Let $f\in \AP(\Aut(\m F))$, $\e>0$. Since $G\cdot f$ is norm-precompact in $\RUCB(\Aut(\m F))$, we can consider the $G$-flow induced on $\overline{G\cdot f}$. By continuity of the action, find a finite substructure $\m A$ of $\m F$ such that for every $g\in \Stab(\m A)$, $\| g\cdot f - f\|_{\infty}<\e$. Consider now the induced $\Stab(\m A)$-flow on the closed convex hull $\overline{\mathrm{co}}(\Stab(\m A)\cdot f)$. This is an affine flow by isometries. By Hahn's fixed point theorem \cite{Gl}*{Chapter III, Section 5}, it admits a fixed point $\chi_{1}$. This is a coloring of $\binom{\m F}{\m A}$ by $\Stab(\m A)$-invariance, and since $\chi_{1}\in \overline{\mathrm{co}}(\Stab(\m A)\cdot f)$, we have $\| \chi_{1} - f\|_{\infty}\leq \e$. At that stage, however, $\chi_{1}$ may not be finite. This can be fixed by repeating the previous argument using the right shift action. Consider the orbit $G\bullet \chi_{1}$. As mentioned above, $G\bullet \chi_{1}$ is also norm-precompact and since $\mathrm{AP}(\Aut(\m F)) \subset \ROB(\Aut(\m F))$, this action is continuous. Hence, there exists a finite substructure $\m Z$ of $\m F$ such that for every $g\in \Stab(\m Z)$, $\| g\bullet \chi_{1} - \chi_{1}\|_{\infty}<\e$. Consider now the induced $\Stab(\m Z)$-flow on $\overline{\mathrm{co}}(\Stab(\m Z)\bullet \chi_{1})$. This is an affine flow by isometries and by Hahn's fixed point theorem, it admits a fixed point $\chi_{2}$. This is still a coloring of $\binom{\m F}{\m A}$, as every point of the orbit $G\bullet \chi_{1}$ is $\Stab(\m A)$-fixed by left shift: for $g\in G$, $h\in \Stab(\m A)$ and $k\in G$, $$ h\cdot(g\bullet \chi_{1})(k)=g\bullet \chi_{1}(h^{-1}k)=\chi_{1}(h^{-1}kg)=h\cdot\chi_{1}(kg)=\chi_{1}(kg)=g\bullet \chi_{1}(k)$$ 

By $\Stab(\m Z)$-invariance, $\chi_{2}$ is in fact constant on all the $\Stab(\m A)\backslash \Aut(\m F)/\Stab(\m Z)$-classes, but by Roelcke-precompactness of $\Aut(\m F)$, there are only finitely many such classes, so that $\chi_{2}$ is finite. Finally, since $\chi_{2}\in \overline{\mathrm{co}}(\Stab(\m B)\bullet \chi_{1})$, we have $\| \chi_{2} - \chi_{1}\|_{\infty}\leq \e$, and therefore $\| \chi_{2} - f \|_{\infty}\leq 2\e$. \end{proof} 

\begin{proof}[Proof of Theorem \ref{thm:MAPStableRP}, $i) \Leftrightarrow ii)$]
Let $\m F$ be a Fra\"iss\'e structure with Roelcke precompact automorphism group. From Proposition \ref{prop:approxAP}, $\Aut(\m F)$ is minimally almost periodic iff every finite coloring in $\mathrm{AP}(\Aut(\m F))$ is constant. Quite clearly, the orbit $\Aut(\m F)\cdot \chi$ is norm-discrete in $\RUCB(\Aut(\m F))$ whenever $\chi$ is a finite coloring. It follows that the only finite colorings in $\mathrm{AP}(\Aut(\m F))$ are those with finite orbit, and all of them are constant iff every $\Aut(\m F)$-invariant equivalence relation on $\binom{\m F}{\m A}$ with finitely many classes is trivial. 
\end{proof}

Note that the Roelcke-precompactness assumption was used to make sure that \emph{finite} colorings are dense in $\mathrm{AP}(\Aut(\m F))$. This is certainly not true in general: Consider an action of $\Z$ on the circle via an irrational rotation $n\cdot \theta = \theta + n\alpha$. This action is isometric, hence equicontinuous, so the map $n\mapsto n\theta$ is almost periodic on $\Z$. It is easy to see that this cannot be $\e$-approximated by a finite almost periodic coloring on $\Z$ for $\e$ small enough.   

\subsection{Minimal almost periodicity, weakly almost periodic colorings and the stable Ramsey property}

\label{subsection:WAP}

As already mentioned in introduction, minimal almost periodicity is equivalent to the formally stronger notion of having a fixed point in any distal flow. The corresponding class of ambits is closed under suprema and factors \cite[Chapter IV, Section 2.27]{dV}. The corresponding compactification is the so-called \emph{maximal group-like compactification of} $G$ \cite[Chapter IV, Section 6.18]{dV}, to which is attached the \emph{distal algebra}. Since this algebra contains the almost periodic one, it could have been interesting to use Theorem \ref{thm:fpRP} to derive a different combinatorial characterization of almost-periodicity than the one obtained using the algebra $\mathrm{AP}(G)$. However, we will not do that for two reasons. The first one is that the description of the distal algebra provided by Theorem \ref{thm:fpRP} does not provide any particularly illuminating way to charactize distal colorings. The second one is that an even more general result can be obtained by considering a still larger algebra of functions, namely, the \emph{weakly almost periodic algebra} $\WAP(G)$, consisting of all those $f\in \RUCB(G)$ such that the closure of $G\bullet f$ is weakly compact in the Banach space $\RUCB(G)$. Note that by the following result of Grothendieck (which we only state here for topological groups), this is equivalent to the fact that $G\cdot f$ is weakly compact: 

\begin{thm}[Grothendieck, \cite{Grothendieck}*{Proposition 7}]

\label{thm:Gro}
Let $G$ be a topological group and $f\in \RUCB(G)$. Then $f\in \WAP(G)$ iff there are no sequences $(g_{m})_{m\in \N}$, $(h_{n})_{n\in \N}$ of elements of $G$ such that $\displaystyle \lim_{m}\lim_{n}f(g_{m}h_{n})$ and $\displaystyle \lim_{n}\lim_{m}f(g_{m}h_{n})$ both exist and are distinct. 
\end{thm} 
 
In addition, by a result of Berglund-Junghenn-Milnes \cite{BJM}*{Chapter III, Lemma 8.8}, we have: $$\WAP(G)=\{ f_{x}:f\in \WAP(G) \wedge x\in W(G)\}$$  

It follows that all minimal subflows of $G\actson W(G)$ are trivial iff $G\actson W(G)$ has a fixed point. This last condition is, in turn, known to be equivalent to minimal almost periodicity for $G$ (for example, this is a consequence of the fact that $B(G)$ is isomorphic to the unique minimal two-sided ideal in $W(G)$ \cite[Chapter III, Section 1.9]{Ru}). 

\begin{prop}
\label{prop:ApproxWAP}
Let $\m F$ be a Fra\"iss\'e structure with Roelcke-precompact automorphism group. Then $\WAP(\Aut(\m F))$ can be approximated by finite colorings. 
\end{prop}

\begin{proof}
See \cite{BYT}*{Proposition 4.7}.
\end{proof}

Applying Theorem \ref{thm:fpRPcpct}, it follows that when it is Roelcke-precompact, $\Aut(\m F)$ is minimally almost periodic iff $\m F$ has the Ramsey property for finite colorings in $\WAP(\Aut(\m F))$. We now proceed like in Section \ref{section:Roelcke} to show that this leads to the equivalence $i)\Leftrightarrow iii)$ in Theorem \ref{thm:MAPStableRP}. To do so, we follow the same scheme as for the proof of Theorem \ref{thm:minRoeckeDefRP}. The first step is to characterize weakly almost periodic colorings combinatorially. Following \cite{BYT}, this can be easily done thanks to Theorem \ref{thm:Gro}. Recall first that according to Proposition \ref{prop:Roelcke coloring}, a finite coloring $\chi$ of $\binom{\m F}{\m A}$ is in $\ROB(\Aut(\m F))$ when there is $\m Z \in \Age(\m F)$ and an embedding $z$ of $\m Z$ in $\m F$ so that $\chi(a)$ depends only on $[a,z]$. We will say then that $\chi$ is \emph{fully determined by $z$} when the converse also holds: $\chi(a)=\chi(a')$ implies $[a,z]=[a',z]$. (In other words, $\chi(a)$ is essentially $[a,z]$.) 

\begin{prop}
\label{prop:WAPStable}
Let $\m F$ be a Fra\"iss\'e structure with Roelcke-precompact automorphism group, $\m A$ a finite substructure of $\m F$ and $\chi$ be a finite coloring of $\binom{\m F}{\m A}$. Assume that $\chi\in \ROB(\Aut(\m F))$ is fully determined by $z\in \binom{\m F}{\m Z}$. Then $\chi\in \WAP(\Aut(\m F))$ iff the pair $(\m A, \m Z)$ is stable. 
\end{prop}

\begin{proof}
We prove that $\chi\notin \WAP(\Aut(\m F))$ iff the pair $(\m A, \m Z)$ is unstable. If $\chi\notin \WAP(\Aut(\m F))$, consider witness sequences $(g_{m})_{m}, (h_{n})_{n}$ provided by Theorem \ref{thm:Gro}. For $m\in \N$, define $a_{m}=\restrict{g_{m}^{-1}}{A}$ and $z_{m}=h_{m}\circ z$. Then for $m,n\in \N$, we have $\chi(g_{m}h_{n})=[a_{m},z_{n}]$. By Roelcke-precompactness of $\Aut(\m F)$, this provides a finite coloring of the pairs of naturals, so by the standard Ramsey's theorem, passing to subsequences, we may assume that there are joint embedding patterns $\tau_{<}$ and $\tau_{>}$ so that  $$\forall m,n\in \N \quad (m<n \Rightarrow [a_{m},z_{n}]=\tau_{<})\wedge (m>n \Rightarrow [a_{m},z_{n}]=\tau_{>})$$

In particular, $\displaystyle \lim_{m}\lim_{n}\chi(g_{m}h_{n})=\tau_{<}$, $\displaystyle \lim_{n}\lim_{m}\chi(g_{m}h_{n})=\tau_{>}$ and by choice of $(g_{m})_{m}$ and $(h_{n})_{n}$, $\tau_{<}$ and $\tau_{>}$ are distinct, witnessing that $(\m A, \m Z)$ is unstable. 

Conversely, assume that $(\m A, \m Z)$ is unstable, as witnessed by sequences $(a_{m})_{m}$ and $(z_{n})_{n}$ and distinct joint embedding patterns $\tau_{<}$ and $\tau_{>}$. By ultrahomogeneity of $\m F$, we can find, for every $m\in \N$, $g_{m}$ and $h_{m}$ so that $a_{m}=\restrict{g_{m}^{-1}}{A}$ and $z_{m}=h_{m}\circ z$.  Then for $m,n\in \N$, we have, as above, $\chi(g_{m}h_{n})=[a_{m},z_{n}]$, so $$\lim_{m}\lim_{n}\chi(g_{m}h_{n})=\tau_{<}\neq \tau_{>} = \lim_{n}\lim_{m}\chi(g_{m}h_{n})$$

Therefore, $\chi\notin \WAP(\Aut(\m F))$. 
\end{proof}

\begin{prop}
Let $\m F$ be a Fra\"iss\'e structure with Roelcke-precompact automorphism group. Then $\m F$ has the Ramsey property for the finite colorings in $\WAP(\Aut(\m F))$ iff for every $\m A, \m B \in \Age(\m F)$, every $\m Z^{1},...,\m Z^{k}\in \Age(\m F)$ so that every pair $(\m A, \m Z^{i})$ is stable, every joint embedding  $\langle \phi,z^{1},...,z^{k}\rangle$ of $\m F$ and $\m Z^{1},...,\m Z^{k}$, there is $b\in \binom{\phi(\m F)}{\m B}$ so that for every $i\leq k$, the coloring $a\mapsto [a,z]$ is constant on $\binom{b(\m B)}{\m A}$.

\end{prop}

\begin{proof}

Assume that $\m F$ has the Ramsey property for colorings in $\WAP(\Aut(\m F))$, fix $\m A, \m B, \m Z^{1},...,\m Z^{k}\in \Age(\m F)$ so that every pair $(\m A, \m Z^{i})$ is stable, and consider a joint embedding of $\m F$ and $\m Z^{1},...,\m Z^{k}$ of the form $\langle \phi,z^{1},...,z^{k}\rangle$. Each coloring defined on $\binom{\m F}{\m A}$ by $a\mapsto [\phi \circ a, z^{i}]$ is finite by Proposition \ref{prop:Roelcke-precompact}, and is in $\WAP(\Aut(\m F))$ by Proposition \ref{prop:WAPStable}. The conclusion follows. 

The converse is an immediate consequence of Proposition \ref{prop:WAPStable}, and of the fact that to check the Ramsey property for colorings in $\WAP(\Aut(\m F))$, it suffices to consider fully determined finite colorings.  
\end{proof}

\begin{prop}
\label{prop:RPWAPStableRP}
Let $\m F$ be a Fra\"iss\'e structure with Roelcke-precompact automorphism group. Then $\m F$ has the Ramsey property for the finite colorings in $\WAP(\Aut(\m F))$ iff $\Age(\m F)$ has the stable Ramsey property. 
\end{prop}

\begin{proof}
The proof is similar to the proof of Proposition \ref{prop:DefinableRP}: The converse implication holds because the stable Ramsey property is a finitization of the Ramsey property holds for colorings in $\WAP(\Aut(\m F))$, while the direct implication is obtained by a compactness argument. 
\end{proof}

\begin{proof}[Proof of Theorem \ref{thm:MAPStableRP}, $i)\Leftrightarrow iii)$]

We have seen in the introduction of the current section that $\Aut(\m F)$ is minimally almost periodic iff $\Aut(\m F)\actson W(\Aut(\m F))$ has a fixed point. By Theorem \ref{thm:fpRPcpct} and Proposition \ref{prop:ApproxWAP}, this happens exactly when $\m F$ has the Ramsey property for colorings in $\WAP(\Aut(\m F))$. Apply then Proposition \ref{prop:RPWAPStableRP}. \end{proof}

\subsection{Remarks}

One of the strengths of the original Kechris-Pestov-Todorcevic correspondence, and of Theorem \ref{thm:KPT} in particular, is its applicability: during the last ten years, it has produced numerous examples of extremely amenable groups and of concrete descriptions of universal minimal flows. It turns out that a similar strategy can be used in order to compute the Bohr compactification of the Roelcke-precompact groups of the form $\Aut(\m F)$. This is suggested by the equivalence $i)\Leftrightarrow ii)$ of Theorem \ref{thm:MAPStableRP}, but was already noticed by Ben Yaacov in \cite{BY} and by Tsankov (personal communication): first, examine whether $i)$ holds by detecting all the invariant equivalence relations with finitely many classes on the sets of the form $\binom{\m F}{\m A}$. If all of those are trivial, the group is minimally almost periodic. If not, determine the non-trivial ones (a task which may not be easy), and the closed subgroup of $\Aut(\m F)$ which fixes all the corresponding classes setwise. At the level of $\m F$, this corresponds to passing to the group $\Aut(\m F^{*})$, where $\m F^{*}$ is the expansion of $\m F$ obtained by naming those classes. This has a natural interpretation from the model-theoretic point of view: it fixes pointwise the algebraic closure of the empty set (in all finite cardinalities). This group is now minimally almost periodic. By Roelcke precompactness, $\m F^{*}$ is a precompact expansion of $\m F$, which means that the quotient $\Aut(\m F)/\Aut(\m F^{*})$ is precompact. By construction, the flow $\Aut(\m F) \actson \widehat{\Aut(\m F)/\Aut(\m F^{*})}$ is minimal and universal for all minimal equicontinuous $\Aut(\m F)$-flows. To show that it is equicontinuous, it suffices to show that $\Aut(\m F^{*})$ is normal in $\Aut(\m F)$, which is easy to check. 

For example, this method can be used to compute the Bohr compactifications for all the groups coming from Fra\"iss\'e graphs and tournaments. Note that in those cases, this may be done using a slightly different method, because the original Kechris-Pestov-Todorcevic correspondence already provides a description of the universal minimal flow as $G\actson \widehat{G/G^{*}}$, where $G^{*}$ is an extremely amenable coprecompact subgroup of $G$. It is then easy to show that the Bohr compactification of $G$ is the (compact) group $G/(G^{*})^{G}$, where $(G^{*})^{G}$ stands for the normal closure of $G^{*}$ in $G$ (for details, see \cite{NVT8}). 

Item $iii)$, on the other hand, should not be thought as a possible way to prove minimal almost periodicity, but rather as a non-trivial combinatorial consequence of it. Of course, to make use of it presupposes an ability to detect stable pairs $(\m A, \m Z)$, a task which can be attacked with model-theoretic tools. 

\section{Proximal flows and proximal colorings}

\label{section:proximal}

The purpose of this section is to concentrate on strong amenability. Ideally, the discussion would have led to analogs of Theorem \ref{thm:fpRoecke}, Theorem \ref{thm:minRoeckeDefRP} and Theorem \ref{thm:MAPStableRP} in the context of proximal flows after the following steps: 1) Description of the corresponding algebra $\A$, 2) Description of the finite colorings in $\A$, 3) Proof of the fact that finite colorings are dense in $\A$, 4) Finitization of the corresponding Ramsey-type statement. While the first two steps can be completed pretty smoothly, this is not the case for the third and fourth, which show some unexpected resistence. This explains the somewhat unsatisfactory form of Theorem \ref{thm:SAProxRP}. 

\subsection{The proximal algebra}

Given a topological group $G$, the class of proximal ambits is closed under suprema and factors \cite[Chapter IV, Section 5.30]{dV}. Since proximality passes to subflows, Theorem \ref{thm:fpRP} applies to the class of proximal flows. Quite surprisingly however, no description of the corresponding C$^{*}$-subalgebra $\mathrm{Prox}(G)$ of $\RUCB(G)$ seems to be available in the literature, so our first task here is to fill this gap thanks to the characterization provided in Theorem \ref{thm:fpRP}: $\mathrm{Prox}(G)$ consists exactly of those $f\in \RUCB(G)$ for which the $G$-flow $G\actson\overline{G\bullet f}$ is proximal (we will call those functions \emph{proximal}). To achieve this, it will be convenient to call a subset $D\subset G^{2}$ \emph{diagonally syndetic} when there is $K\subset G$ finite so that $$G^{2}=K\cdot D \ (=\bigcup_{k\in K}k\cdot D)$$ where $g\cdot(g_{1},g_{2})$ refers to the diagonal action: $g\cdot(g_{1},g_{2})=(gg_{1}, gg_{2})$. This definition is of course modeled on the standard concept of syndetic subset of $G$, where $S\subset G$ is syndetic when there is a finite $K\subset G$ so that $G=K\cdot S$.

For a $G$-flow $G\actson X$, $x_{1}, x_{2}\in X$, $U\subset X^{2}$, define the set $P(x_{1},x_{2},U)$ as: $$P(x_{1}, x_{2},U):=\{ (g_{1}, g_{2})\in G^{2}: (g_{1}\cdot x_{1}, g_{2}\cdot x_{2})\in U\}$$  

\begin{prop}
Let $G\actson X$ be a $G$-flow, $x_{1}, x_{2}\in X$. TFAE: 
\begin{enumerate}
\item[i)] For every entourage $U$, the set $P(x_{1},x_{2},U)$ is diagonally syndetic.  
\item[ii)] For every $(y_{1}, y_{2})\in \overline{G\cdot x_{1}\times G\cdot x_{2}}$, every entourage $U$, there exists $g\in G$ so that $g\cdot(y_{1}, y_{2})\in U$. 
\end{enumerate}
\end{prop}

\begin{proof}
$i)\Rightarrow ii)$: Fix $(y_{1}, y_{2})\in \overline{G\cdot x_{1}\times G\cdot x_{2}}$ and $U$ an entourage of the diagonal in $X$, which we may take compact. We will show that there is a finite set $K\subset G$ so that $G\cdot x_{1}\times G\cdot x_{2} \subset K\cdot U$. This is sufficient:  passing to closures $\overline{G\cdot x_{1}\times G\cdot x_{2}}\subset \overline{K\cdot U}=K\cdot U$, so $(y_{1}, y_{2})\in k\cdot U$ for some $k\in K$, so $g=k^{-1}$ satisfies $g\cdot(y_{1}, y_{2})\in U$, as required. To prove the existence of $K$: $P(x_{1},x_{2},U)$ is diagonally syndetic, so we can write $G^{2}=K\cdot P(x_{1},x_{2},U)$ for some finite $K\subset G$. Now, for $g_{1}, g_{2}\in G$, we have $(g_{1}, g_{2})=k\cdot(h_{1}, h_{2})$ for some $k\in K$ and $(h_{1}, h_{2})\in P(x_{1},x_{2},U)$, so $(g_{1}\cdot x_{1}, g_{2}\cdot x_{2})=k\cdot (h_{1}\cdot x_{1}, h_{2}\cdot x_{2})\in K\cdot U$. 

$ii) \Rightarrow i)$: Fix $U$ an open entourage of the diagonal in $X$. By assumption, $\overline{G\cdot x_{1}\times G\cdot x_{2}}\subset \bigcup_{g\in G}g\cdot U$, so by compactness, there $K\subset G$ finite such that $\overline{G\cdot x_{1}\times G\cdot x_{2}}\subset \bigcup_{g\in K}g\cdot U$. Now, $$G^{2}=P(x_{1}, x_{2}, \overline{G\cdot x_{1}\times G\cdot x_{2}})=P(x_{1}, x_{2}, \bigcup_{g\in K}g\cdot U)=\bigcup_{g\in K} g\cdot P(x_{1},x_{2},U) \qedhere$$\end{proof}

As a direct corollary: 

\begin{prop}
\label{prop:prox}
Let $G\actson X$ be a $G$-flow, $x\in X$. Then $\overline{G\cdot x}$ is proximal iff for every entourage $U$ of the diagonal in $X$, the set $P(x,x,U)$ is diagonally syndetic. 
\end{prop}

Specializing this to the $G$-flow $G\actson \overline{G\bullet f}$, 
we directly obtain: 

\begin{prop}
\label{prop:ProxAlg}
Let $f\in \RUCB(G)$. Then $f\in \mathrm{Prox}(G)$ iff for every finite $F\subset G, \e>0$, there exists a finite $K\subset G$ such that for every $(g_{1}, g_{2})\in G^{2}$, there exists $k\in K$ such that $g_{1}\bullet f$ and $g_{2}\bullet f$ are equal up to $\e$ on $Fk$. 
\end{prop}

\subsection{Proximal colorings, fixed points in zero-dimensional proximal flows and proximal Ramsey property}

\label{section:ProxColorings}

We now turn to a description of the colorings in $\mathrm{Prox}(\Aut(\m F))$ and to a proof of Theorem \ref{thm:SAProxRP}. Specializing Proposition \ref{prop:ProxAlg} to the case where $G=\Aut(\m F)$ with $\m F$ Fra\"iss\'e and $f$ a finite coloring, we obtain: 

\begin{prop}
Let $\m F$ be a Fra\"iss\'e structure, $\chi$ be a finite coloring of $\binom{\m F}{\m A}$. Then $\chi \in \mathrm{Prox}(\Aut(\m F))$ iff for every $\m D \in \Age(\m F)$, there are copies $\m D_{1},...,\m D_{k}$ of $\m D$ in $\m F$ such that for every $(g_{1}, g_{2})\in G^{2}$, there is $i\leq k$ such that $g_{1}\cdot \chi = g_{2}\cdot \chi$ on $\binom{\m D_{i}}{\m A}$. 
\end{prop}

Observing now that $\m D_{1},...,\m D_{k}$ are contained in some finite $\m E$, we obtain: 

\begin{prop}
\label{prop:ProxColorings}
Let $\m F$ be a Fra\"iss\'e structure and $\chi$ be a finite coloring of $\binom{\m F}{\m A}$. Then $\chi \in \mathrm{Prox}(\Aut(\m F))$ iff $\chi$ is proximal.
\end{prop}

\begin{prop}
Let $\m F$ be a Fra\"iss\'e structure. Then $\m F$ has the Ramsey property for colorings in $\mathrm{Prox}(\Aut(\m F))$ iff $\m F$ has the proximal Ramsey property. 
\end{prop}

\begin{proof}
The Ramsey property for colorings in $\mathrm{Prox}(\Aut(\m F))$ refers to finite collections of finite proximal colorings, while the proximal Ramsey property only refers to one such coloring. The direct implication is therefore obvious. For the converse, notice that given a finite set $\chi_{1},...,\chi_{l}$ of finite proximal colorings, the $\Aut(\m F)$-ambit $\bigvee_{i=1}^{l} (\overline{\Aut(\m F)\cdot\chi_{i}},\chi_{i})$ is proximal. It follows that the coloring $a\mapsto (\chi_{i}(a))_{1\leq i\leq l}$ is also proximal, so by the proximal Ramsey property, it is constant on $\binom{b(\m B)}{\m A}$ for some $b$. Clearly, each $\chi_{i}$ is then constant on $\binom{b(\m B)}{\m A}$, witnessing that $\m F$ has the Ramsey property for colorings in $\mathrm{Prox}(\Aut(\m F))$. 
\end{proof}

\begin{proof}[Proof of Theorem \ref{thm:SAProxRP}]
According to Theorem \ref{thm:fpRP}, every zero-dimensional proximal $\Aut(\m F)$-flow has a fixed point iff $\m F$ has the Ramsey property for colorings in $\mathrm{Prox}(\Aut(\m F))$. By the previous proposition, this is equivalent to the proximal Ramsey property. 
\end{proof}

\subsection{Remarks}

\label{subsection:RksProx}

The difficulty to prove that finite colorings are dense in the proximal algebra is the main obstacle to a more satisfactory form of Theorem \ref{thm:SAProxRP}, and it is reasonable to wonder where it is coming from. Can this be solved by adding an extra natural topological hypothesis on $\Aut(\m F)$, which would play the r\^{o}le that Roelcke precompactness played for distal flows? Note that even if this were possible, the relevance of the present approach as an effective method to prove strong amenability by combinatorial means looks rather questionable, as the proximality condition on colorings does not seem to make it particularly easy to deal with in practice. Note also that, in the same vein, it would be interesting to find a topological property that ensures that the proximal universal minimal flow of $\Aut(\m F)$ is metrizable. (This should probably be equivalent to the fact that $\Aut(\m F)$ contains a co-precompact strongly amenable closed subgroup.) 

In a slightly different spirit: Assume that a Polish group $G$ is minimally almost periodic and strongly amenable. Is $G$ necessarily extremely amenable? The answer is positive when the universal minimal flow of $G$ is metrizable (see \cite{NVT8}) but the general case remains open. In fact, even when $G$ is assumed to be monothetic (= contains a dense cyclic subgroup), this is the content of a famous open problem of Glasner (see \cite{Gl1}, as well as Pestov's contribution in \cite{OT} for a detailed account about it).   

\section{Strongly proximal flows and amenability}

\label{section:stronglyproximal}

Following Furstenberg, recall that a flow is \emph{strongly proximal} when the affine flow it induces on the space of Borel probability measures is proximal. These flows are well-behaved in the sense that they satisfy the hypotheses of Theorem \ref{thm:fpRP}. In addition, the fixed point property on this class is equivalent to being amenable, which, in principle, makes amenability approachable by the general method of the present paper. However, in practice, the obstructions that appeared with proximal flows in the previous section also appear when dealing with strongly proximal flows. In addition, because of a lack of a characterization of strong amenability in terms of syndetic sets in the spirit of Proposition \ref{prop:prox}, no characterization of the strongly proximal algebra parallel to Proposition \ref{prop:ProxAlg} is available at the moment. For those reasons, the specialization of Theorem \ref{thm:fpRP} to amenability and strongly proximal flows will not be detailed further here. 

Nevertheless, there does exist a Ramsey-theoretic characterization of amenability, provided by Theorem \ref{thm:AConvexRP}. This result is originally due to Moore \cite{M} and to Tsankov \cite{Ts1}. Both proofs are rather similar, and pretty close to the following one, which is in the spirit of the rest of the paper. 

\begin{proof}[Proof of Theorem \ref{thm:AConvexRP}]

The starting point is the following characterization of amenability: A topological group $G$ is amenable iff every $G$-flow admits an invariant (Borel probability) measure. Because the Samuel compactification $S(G)$ maps onto any minimal $G$-flow, this is equivalent to the existence of a fixed point in $\mathrm{Prob}(S(G))$, the set of all Borel probability measures on $S(G)$. This set is compact and convex, and it admits a fixed point iff the following statement $(*)$ holds: for every finite family $\mathcal F$ of continuous affine maps on $\mathrm{Prob}(S(G))$, every $\varepsilon >0$, every finite $H\subset G$, there exists $\mu \in \mathrm{Prob}(S(G))$ which is fixed up to $(\mathcal F, \e, H)$, i.e. every $f\in \mathcal F$ is constant on $H\cdot\mu$ up to $\varepsilon$. 

Now, since $G$ is dense in $S(G)$ and the finitely supported measures on $S(G)$ are dense in $\mathrm{Prob}(S(G))$, the above $\mu$ can be replaced by a finite convex linear combination $\sum_{i=1}^{n}\lambda_{i}\delta_{g_{i}}$. Next, because $S(G)$ is the set of extreme points in $ \mathrm{Prob}(S(G))$, every element of $\mathcal F$ is nothing else than the natural affine extension of its restriction to $S(G)$. This, in turn, is just an element of $C(S(G))=\RUCB(G)$. In other words, $(*)$ is equivalent to: for every finite $\mathcal F \subset \RUCB(G)$, every $\varepsilon >0$, every finite $H\subset G$, there exists a convex linear combination $\lambda_{1},..., \lambda_{n}$ and $g_{1},...,g_{n}\in G$ such that for every $f\in \mathcal F$, the map $$h\mapsto f(h\cdot \sum_{i=1}^{n}\lambda_{i}\delta_{g_{i}})=\sum_{i=1}^{n}\lambda_{i}f(hg_{i})$$ is constant on $H$ up to $\e$. Note that without loss of generality, we may assume that $\mathcal F$ consists of one single $f\in \RUCB(G)$. 

When $G$ is of the form $\Aut(\m F)$ for some Fra\"iss\'e structure $\m F$, this discretizes (in the spirit of Section \ref{section:colorings}) as: for every $\m A, \m B \in \Age(\m F)$, every $\e>0$, and every finite coloring $\chi$ of $\binom{\m F}{\m A}$, there is a finite convex linear combination $\lambda_{1},...,\lambda_{n}$, and $b_{1},...,b_{n}\in \binom{\m F}{\m B}$ such that the coloring $\displaystyle a\mapsto \sum_{i=1}^{n}\lambda_{i}\chi(b_{i}\circ a)$ is constant on $\binom{\m B}{\m A}$. 

This, in turn, is equivalent to the convex Ramsey property via a standard compactness argument. 
\end{proof}

As indicated in the introduction, the practical use of Theorem \ref{thm:AConvexRP} is so far limited. There are promising exceptions, as the papers \cite{GKP} by Gadhernezhad, Khalilian and Pourmahdian, and \cite{EG} by Etesami and Gadhernezhad, do make use of it to prove that certain automorphism groups of the form $\Aut(\m F)$, where $\m F$ is a so-called \emph{Hrushovski structure}, are not amenable. Nevertheless, there is presently no significant instance where Theorem \ref{thm:AConvexRP} can be used to prove that some group is amenable. There are substantial results regarding amenability of groups of the form $\Aut(\m F)$ (see for example \cites{AKL, PS}), but all of them rest on an explicit description of the universal minimal flow, as well as on an analysis of the invariant measures on this flow. This method, in turn, imposes severe restrictions on the groups under consideration. 

Quite interestingly though, the use of the convex Ramsey property to characterize amenability naturally leads to the following question: Is there a characterization of strong amenability in similar terms? Once again, the answer is positive when the universal flow is metrizable (see \cite{MNT}), but the general answer remains unknown, due to the lack of a general characterization of strong amenability in terms of existence of invariant measures.

\subsection*{Acknowledgements}

This paper has benefited from numerous discussions with several people. I would like to particularly thank Ita\"i Ben Yaacov regarding the Roelcke and the Bohr compactifications, as well as the notion of stability; Eli Glasner regarding the r\^{o}le of ambits, as opposed to flows; Michael Megrelishvili and Vladimir Pestov regarding the notion of point-universality; Julien Melleray for his sharpness to detect sloppy arguments; Todor Tsankov for the helpful references concerning unitary representations; Benjy Weiss regarding proximal flows; Sylvie Benzoni-Gavage, Isabelle Chalendar and Xavier Roblot, for hosting at the Institut Camille Jordan; and finally the anonymous referee for her/his very careful revision of the paper. Her/His suggestions and comments led to the correction of several mistakes, and substantially improved the quality of the paper.

\bibliographystyle{amsalpha}
\bibliography{Bib17Jan}
\end{document}